\documentclass{article}
\usepackage{graphicx} 

\usepackage{amssymb}
\usepackage{amsmath}
\usepackage{bbm}
\usepackage{geometry}
\usepackage{authblk}
\usepackage{amsthm}
\usepackage{hyperref}
\hypersetup{colorlinks=true,linkcolor=blue,citecolor=red}

\theoremstyle{plain}
\newtheorem{theorem}{Theorem}[section]
\newtheorem{proposition}{Proposition}[section]
\newtheorem{lemma}{Lemma}[section]

\theoremstyle{definition}
\newtheorem{definition}{Definition}[section]
\newtheorem{claim}{Claim}[section]
\newtheorem{problem}{Problem}[section]
\newtheorem{case}{Case}
\newtheorem{example}{Example}[section]

\theoremstyle{remark}
\newtheorem{remark}{Remark}[section]

\title{Measuring the convexity of compact sumsets with the Schneider non-convexity index}
\author{Mark Meyer
\thanks{Email: memeyer@shockers.wichita.edu, Wichita, KS, USA}}
\affil{Wichita State University department of mathematics}

\begin{document}
\maketitle

\begin{abstract}
    In recent work, Franck Barthe and Mokshay Madiman introduced the concept of the Lyusternik region, denoted by $\Lambda_{n}(m)$, to better understand volumes of sumsets. They gave a characterization of $\Lambda_{n}(2)$ (the volumes of compact sets in $\mathbb{R}^n$ when at most $m=2$ sets are added together) and proved that Lebesgue measure satisfies a fractional superadditive property. We attempt to imitate the idea of the Lyusternik region by defining a region based on the Schneider non-convexity index function, which was originally defined by Rolf Schneider in 1975. We call this region the Schneider region, denoted by $S_{n}(m)$. In this paper, we will give an initial characterization of the region $S_{1}(2)$ and in doing so, we will prove that the Schneider non-convexity index of a sumset $c(A_1+A_2)$ has a best lower bound in terms of $c(A_1)$ and $c(A_2)$. We will pose some open questions about extending this lower bound to higher dimensions and large sums. We will also show that, analogous to Lebesgue measure, the Schneider non-convexity index has a fractional subadditive property. Regarding the Lyusternik region, we will show that when the number of sets being added is $m\geq3$, that the region $\Lambda_{n}(m)$ is not closed, proving a new qualitative property for the region.
\end{abstract}

\textbf{Classification code:} 52A40\\
\textbf{Data availability statement:} Data sharing not applicable to this article as as no data sets were generated or analysed during the current study.

{\hypersetup{linkcolor=black}
\tableofcontents}
\section{Introduction}
\subsection{The Schneider non-convexity index}
In 1975 Rolf Schneider introduced a measure of non-convexity \cite{schneider1}. For a set $A\subseteq\mathbb{R}^n$ the Schneider non-convexity index of $A$ is defined by
\begin{equation*}
    c(A):=\inf\{\lambda\geq0:A+\lambda\text{conv}(A)\text{ is convex}\}.
\end{equation*}
In the case that $A$ is the empty set, $c(\varnothing)=0$. In the original paper the following theorem was proved.
\begin{theorem}[Schneider \cite{schneider1}]\label{Shneider_Thm}
    Let $A\subseteq\mathbb{R}^n$ be a compact set. Then $0\leq c(A) \leq n$. The left inequality is equality if and only if $A$ is convex, and the right inequality is equality if and only if $A$ consists of exactly $n+1$ affinely independent points.
\end{theorem} 
Theorem \ref{Shneider_Thm} hints that the Schneider non-convexity index makes a good measure of non-convexity in the sense that the only sets that can have an index of zero are convex sets.

A recent review \cite{fradelizi1} contains results about functions which can be thought of as measures of non-convexity. In particular, the authors review the Schneider non-convexity index, and prove the following upper bound for the sum of three sets.
\begin{theorem}[Fradelizi et al. \cite{fradelizi1}]\label{strong bound}
    Let $A,B,C$ be subsets of $\mathbb{R}^{n}$. Then
    \begin{equation}\label{strong_bound}
        c(A+B+C)\leq\max\{c(A+B),c(B+C)\}.
    \end{equation}
\end{theorem}
In section \ref{fractionally_subadditive} we will show that this upper bound can be used to prove a larger set of inequalities for the Schneider non-convexity index.

\subsection{The Lyusternik region}\label{lyusternik}
The Lyusternik region, introduced in the paper \cite{barthe1} is defined as follows. 
\begin{definition}[The Lyusternik region]
    For positive integers $m,n$, let $K_n(m)$ denote the set of all $m$-tuples $A=(A_1,\dots,A_m)$ of non-empty, compact sets in $\mathbb{R}^n$. Let $2^{[m]}$ denote the collection of all subsets of $[m]:=\{1,\dots,m\}$. For each $A\in K_n(m)$, define the function $v_A:2^{[m]}\rightarrow\mathbb{R}$ by
\begin{equation*}
    v_A(S):=\left|\sum_{i\in S}A_i\right|,
\end{equation*}
where $|A|$ denotes the Lebesgue measure of the set $A$. The $(n,m)$-Lyusternik region is defined by
\begin{equation*}
    \Lambda_{n}(m):=\{v_A: A\in K_n(m)\}.
\end{equation*}
The set $\Lambda_n(m)$ can be viewed as a subset of $\mathbb{R}^{2^m}$ in the following way. Choose an ordering $S_1,\dots,S_{2^m}$ of $2^{[m]}$. Then the function $v_A$ corresponds to the vector $(v_A(S_1),\dots,v_A(S_{2^m}))$. For example, if $m=2$, then the function $v_A$ will be associated with the vector $(0,|A_1|,|A_2|,|A_1+A_2|)$. In fact, if we wanted we could drop the first component corresponding to the empty set, since that component will always be $0$.
\end{definition}
What follows is some terminology which is of interest in the study of the Lyusternik region.
\begin{enumerate}
    \item \textbf{Fractional partitions:} For an integer $m\geq2$, a set $\varnothing\neq\mathcal{G}\subseteq 2^{[m]}$ of subsets of $[m]$ is called a \textit{hypergraph}. If $\beta:\mathcal{G}\rightarrow[0,\infty)$ is some function, then the pair $(\mathcal{G},\beta)$ is called a \textit{weighted hypergraph}. If, in addition, the function $\beta$ (which can be thought of as an assignment of weights to each of the sets $S\in\mathcal{G}$) satisfies: $\forall i\in[m]$, $\sum_{S\in\mathcal{G}:i\in S}\beta(S)=1$, then the pair $(\mathcal{G},\beta)$ is called a \textit{fractional partition} of $[m]$.
    \item \textbf{Fractional superadditive/subadditive set functions:} A function $f:2^{[m]}\rightarrow[0,\infty)$ is called \textit{fractionally superadditive} if for each $T\subseteq[m]$,
\begin{equation*}
    f(T)\geq\sum_{S\in\mathcal{G}}\beta(S)f(S)
\end{equation*}
is true for any fractional partition $(\mathcal{G},\beta)$ of $T$. If the inequality sign is flipped, then the set function is called \textit{fractionally subadditive}. We will use the notation $\Gamma_{FSA}(m)$ to denote the set of all fractionally superadditive functions $f$ as defined above for which $f(\varnothing)=0$.
    \item \textbf{Supermodular/submodular set functions:} A function $f:2^{[m]}\rightarrow[0,\infty)$ is called $\textit{supermodular}$ if for any $S,T\subseteq[m]$,
\begin{equation*}
    f(S\cup T)+f(S\cap T)\geq f(S)+f(T).
\end{equation*}
If the inequality sign is flipped, then the function is called \textit{submodular}. We will use the notation $\Gamma_{SM}(m)$ to denote the set of all supermodular functions $f$ as defined above such that $f(\varnothing)=0$.
\end{enumerate}  
The following result, which is proved in \cite{barthe1}, summarizes what is known about the Lyusternik region and its relation with $\Gamma_{FSA}(m)$ and $\Gamma_{SM}(m)$.
\begin{theorem}[Barthe and Madiman \cite{barthe1}]\label{lyusternik_region}
    The following statements are true. 
    \begin{enumerate}
        \item In the case that $m=2$, and $n\geq1$, the Lyusternik region is characterized by 
    \begin{equation*}
        \Lambda_{n}(2)=\{(0,a,b,c)\in\mathbb{R}^{4}_{+}:c\geq(a^{\frac{1}{n}}+b^{\frac{1}{n}})^n\}.
    \end{equation*}
    In fact, we have the relation
    \begin{equation*}
        \Lambda_{1}(2)=\Gamma_{SM}(2)=\Gamma_{FSA}(2).
    \end{equation*}
        \item If $m\geq3$, then $\Lambda_{n}(m)\subsetneq\Gamma_{FSA}(m)$.
        \item If $m\geq3$, then $\Lambda_n(m)$ and $\Gamma_{SM}(m)$ have nonempty intersection, but neither is a subset of the other.
    \end{enumerate}
\end{theorem}
The first statement of Theorem \ref{lyusternik_region} shows that the Lyusternik region is completely determined by the Brunn-Minkowski-Lyusternik inequality for compact sets \cite{lyusternik1} when the number of sets being added is $m=2$. We also note that related to the Lyusternik region is the Stam region \cite{madiman1}.

\subsection{Outline of the paper}
The motivation of the results in this paper came from the papers \cite{barthe1,fradelizi1,schneider1}. It is natural to attempt to imitate the idea of the Lyusternik region with one of the measures of non-convexity, and the Schneider non-convexity index is a natural candidate for this. What follows is an outline of the new contributions this paper adds to the study of the Schneider non-convexity index and the Lyusternik region.
\begin{enumerate}
    \item \textbf{Fractional subadditivity of $c$:} In Section \ref{fractionally_subadditive} we will use the inequality in Theorem \ref{strong bound} along with a simplification from the paper \cite{barthe1} to prove that the Schneider non-convexity index is fractionally subadditive.
    \item \textbf{The Schneider region:} In Section \ref{schneider_region} we will define the Schneider region, denoted by $S_n(m)$, which is defined in an analogous way to the Lyusternik region, except that the underlying set function is the Schneider non-convexity index instead of Lebesgue measure. We will give a characterization of the region $S_1(2)$, which turns out to be a difficult problem and actually leads to the main result of this paper which is explained in the next item.
    \item \textbf{A best lower bound:} In Section \ref{best_lower_bound_for_cartesian_product_sets}, we will prove that if $A_1$ and $A_2$ are non-empty, compact sets in $\mathbb{R}$, then the index of the sumset $c(A_1+A_2)$ has a best lower bound in terms of the indices $c(A_1)$ and $c(A_2)$. This lower bound can then be extended to a lower bound for Cartesian product sets in $\mathbb{R}^n$, which is explained in Section \ref{cartesian_bound}. It turns out that it is tedious to prove this is a lower bound, even in the one-dimensional case. We summarize some of the unanswered questions in Section \ref{open_problems}.
    \item \textbf{Closure properties of the Lyusternik region:} In Section \ref{closure_properties} we will define a one-dimensional fractal set, and then use it to characterize a small piece of the region $\Lambda_{n}(3)$. It will follow that the region $\Lambda_{n}(m)$ is not a closed set when $m\geq3$, which contrasts with the known characterization for $\Lambda_{n}(2)$.
\end{enumerate}

\subsection{Related research}
For work related to fractional superadditivity, see \cite{bobkov1,barthe1,fradelizi2}. These papers cover a conjecture made by Bobkov et al. of a fractional generalization the Brunn-Minkowski inequality. For the case where the sets are compact, but not convex the conjecture was resolved in dimension 1, but counter examples were found in dimension $n\geq7$. Around the same time a special case of the conjecture was verified in dimensions $2$ and $3$ for star-shaped sets. In fact, before the fractional Brunn-Minkowski conjecture, it was shown that the power entropy of a random variable is fractionally superadditive \cite{artstein1,madiman1,madiman2}.  

For references on supermodularity, we note the paper \cite{ollagnier1} where it is proved that supermodularity implies fractional superadditivity (under the condition that the empty set is mapped to $0$), and in \cite{fradelizi3}, it is proved that volume satisfies a more generalized supermodularity property known as m-supermodularity, where the higher order supermodularity seems to have originated in \cite{foldes1}.

An area of study related to this study of the Schneider non-convexity index is the study of the convexification property that Minkowski summation has on compact sets. This fundamental observation is given by the Shapley-Folkman theorem \cite{emerson1,starr1,starr2} which states that the $k$th set average $A(k):=\frac{1}{k}\sum_{i=1}^{k}A$ of a compact set $A$ converges in volume deficit (if $|A(k)|>0$ for some $k$ when $\textup{card}(A)>1$) and the Hausdorff metric to $\textup{conv}(A)$. Recently this property has been studied in the infinite dimensional case \cite{artstein2}.

\subsection{Acknowledgements}
I thank Robert Fraser and Buma Fridman for many helpful discussions on this topic. I especially thank Robert Fraser for taking the time to carefully read through this paper and for offering helpful corrections and advice. The author of this paper was supported in part by the National Science Foundation LEAPS - Division of Mathematical Sciences Grant No. 2316659.

\section{The best lower bound}\label{best_lower_bound_for_cartesian_product_sets}
\subsection{Notation}
\begin{enumerate}
\item \textbf{A useful formula:} For a compact set $A\subseteq\mathbb{R}$, we define the largest gap in $A$ by
\begin{equation*}
    G(A):=\inf\{\varepsilon\geq0:A+[0,\varepsilon]\textup{ is convex}\}.
\end{equation*}
The diameter of $A$ is defined by
\begin{equation*}
    \textup{diam(A)}:=|\textup{conv}(A)|.
\end{equation*}
For the one-dimensional case, the Schneider non-convexity index of the set $A$ can be computed with the convenient formula
\begin{equation*}
    c(A)=\frac{G(A)}{\textup{diam}(A)}.
\end{equation*}
To verify that this formula is correct, let $A$ be compact, and translated so that $\min(A)=0$. Then we can write $\textup{conv}(A)=[0,a]$, where $a>0$. Observe that
\begin{equation*}
    \begin{split}
        c(A)&=\inf\{\lambda\geq0:A+\lambda\textup{conv}(A)\text{ is convex}\}\\
        &=\inf\{\lambda\geq0:A+a[0,\lambda]\text{ is convex}\}\\
        &=\inf\left\{\frac{\varepsilon}{a}\geq0:A+a\left[0,\frac{\varepsilon}{a}\right]\text{ is convex}\right\}\\
        &=\frac{1}{a}\inf\{\varepsilon\geq0:A+[0,\varepsilon]\text{ is convex}\}\\
        &=\frac{G(A)}{\textup{diam}(A)}.
    \end{split}
\end{equation*}
\item We will define the function $r:[0,1]\rightarrow[0,1]$ by
\begin{equation*}
    r(c)=\frac{1-c}{1+c}.
\end{equation*}
To simplify notation, if $\sigma\in S_k$ (the group of permutations on $[k]:=\{1,\dots,k\}$), and $\{A_j\}_{j=1}^{k}$ are compact sets, then we will define
\begin{equation*}
    r_{\sigma(j)}:=r\left(c(A_{\sigma(j)})\right).
\end{equation*}
\item \textbf{R-sets and L-sets:} For $0\leq \alpha<1$ and $\alpha^{\prime}\geq0$, suppose $A$ is of the form $A=[0,\alpha]\cup[1,1+\alpha^{\prime}]$. If $\alpha>\alpha^{\prime}$, then $A$ is called an \textit{L-set} (i.e the left interval is longer than the right interval). If $\alpha<\alpha^{\prime}$, then $A$ is called an \textit{R-set} (i.e the right interval is longer than the left interval). In the case that $\alpha=\alpha^{\prime}$, $A$ is called a \textit{balanced} set.
\item If $m\geq1$ is some integer and $A_1,\dots,A_m$ are sets in $\mathbb{R}^n$, then for a set $S\subseteq [m]$ we will define $A_{S}:=\sum_{i\in S}A_i$.
\item If $a$ and $b$ are real numbers, such that either $\textup{sign}(a)=\textup{sign}(b)$ or $a=b=0$, then we will use the notation $a\approx b$. It will be convenient to use this notation with the derivative of a rational function $f(x)=\frac{p(x)}{q(x)}$. The derivative satisfies $f^{\prime}(x)\approx q(x)p^{\prime}(x)-p(x)q^{\prime}(x)$. Also, if $\gamma(n_1,n_2)$ is a rational function of two variables, then we will use the notation $D_j\gamma(n_1,n_2)$ to denote the partial derivative of $\gamma$ with respect to the $j$th coordinate evaluated at $(n_1,n_2)$.
\end{enumerate}
\subsection{Compact sets in one-dimension}
The objective of this subsection is to prove 
\begin{theorem}\label{Lower_bound}
    Let $A_1$ and $A_2$ be non-empty, compact sets in $\mathbb{R}$. Then
    \begin{equation*}
        c(A_1+A_2)\geq\max\left(0,\min_{\sigma\in S_2}\frac{1-r_{\sigma(1)}-2r_{\sigma(2)}}{3+r_{\sigma(1)}+2r_{\sigma(2)}}\right).
    \end{equation*}
    This is the best lower bound in the sense that if $c_1,c_2\in[0,1]$, then
    \begin{equation*}
        \min_{\substack{A\in\mathcal{K}_1(2)\\c(A_j)=c_j\\j\in[2]}}c(A_1+A_2)=\max\left(0,\min_{\sigma\in S_2}\frac{1-r(c_{\sigma(1)})-2r(c_{\sigma(2)})}{3+r(c_{\sigma(1)})+2r(c_{\sigma(2)})}\right).
    \end{equation*}
    That is, minimizing $c(A_1+A_2)$ over all the compact sets $A_1,A_2$ with fixed convexity indices $c_1,c_2$ respectively, achieves exactly the lower bound.
\end{theorem}
The next result gives a characterization of L-sets and R-sets in terms of the function $r$.
\begin{lemma}\label{LR-set}
Let $A$ be a nonempty, compact set in $\mathbb{R}$, let $c:=c(A)$ denote the convexity index of $A$, and define the number
\begin{equation*}
    r:=\frac{1-c}{1+c}.
\end{equation*}
Then the following characterizations hold.
\begin{enumerate}
    \item If $A$ is an $L$-set, then we can write
    \begin{equation*}
        A=\left[0,r+n\left(\frac{1-r}{2}\right)\right]\cup\left[1,1+r-n\left(\frac{1+r}{2}\right)\right],
    \end{equation*}
    where $n>0$ is a real number which satisfies $n(\frac{1+r}{2})\leq r$.
    \item If $A$ is an $R$-set, then we can write
    \begin{equation*}
        A=\left[0,r-n\left(\frac{1-r}{2}\right)\right]\cup\left[1,1+r+n\left(\frac{1+r}{2}\right)\right],
    \end{equation*}
    where $n>0$ is a real number which satisfies $n(\frac{1-r}{2})\leq r$.
    \item If $A$ is a balanced set, then we can write $A=\left[0,r\right]\cup\left[1,1+r\right]$.
\end{enumerate}
\end{lemma}
\begin{proof}
    We will prove this result for the case that $A$ is an L-set. If $A$ is an R-set, then the proof will be almost identical, and if $A$ is a balanced set, just set $n=0$ in the characterization for either of the first two parts. When $A$ is an L-set, write $A=[0,\alpha]\cup[1,1+\alpha^{\prime}]$, where $\alpha>\alpha^{\prime}$. Then we can write $\alpha=\alpha^{\prime}+n$, where $n:=\alpha-\alpha^{\prime}$. Then $A$ can be written as $A=[0,\alpha^{\prime}+n]\cup[1,1+\alpha^{\prime}]$. The convexity index of $A$, call it $c$, is given by
    \begin{equation*}
        c=\frac{1-\alpha^{\prime}-n}{1+\alpha^{\prime}}.
    \end{equation*}
    This implies that 
    \begin{equation*}
        \alpha^{\prime}=r-n\left(\frac{1+r}{2}\right).
    \end{equation*}
    The characterization for the L-set follows.
\end{proof}
The purpose of the next result is to verify that filling in the gaps of sets $A$ and $B$ to achieve new sets $A^{\prime}$ and $B^{\prime}$ will give a sumset $A^{\prime}+B^{\prime}$ that is closer to being convex than the sumset $A+B$.
\begin{lemma}\label{filling_gaps}
    Let $A$ and $B$ be nonempty, compact sets in $\mathbb{R}^n$. If $A\subseteq A^{\prime}\subseteq\textup{conv}(A)$ and $B\subseteq B^{\prime}\subseteq\textup{conv}(B)$, then $c(A^{\prime}+B^{\prime})\leq c(A+B)$.
\end{lemma}
\begin{proof}
    Using the notation given in the statement, if $\lambda\geq c(A+B)$, then
    \begin{equation*}
    \begin{split}
        A+B+\lambda\textup{conv}(A+B)&=\textup{conv}(A)+\textup{conv}(B)+\lambda\textup{conv}(A+B)\\
        &\supseteq A^{\prime}+B^{\prime}+\lambda\textup{conv}(A^{\prime}+B^{\prime})\\
        &\supseteq A+B+\lambda\textup{conv}(A+B).
    \end{split}
    \end{equation*}
    It follows that 
    \begin{equation*}
        A+B+\lambda\textup{conv}(A+B)=A^{\prime}+B^{\prime}+\lambda\textup{conv}(A^{\prime}+B^{\prime}).
    \end{equation*}
    Since the set on the left is convex, the set on the right must also be convex. Therefore $\lambda\geq c(A^{\prime}+B^{\prime})$, from which it follows that $c(A^{\prime}+B^{\prime})\leq c(A+B)$.
\end{proof}
The proof of Theorem \ref{Lower_bound} will be broken into four different cases, which are verified in the next four lemmata.
\begin{lemma}\label{L-set,R-set}
    Let $A_1$ be an L-set and $A_2$ be an R-set such that 
    \begin{equation}\label{positive_assumption}
        \min_{\sigma\in S_2}\left(1-r_{\sigma(1)}-2r_{\sigma(2)}\right)>0.
    \end{equation}
    Then for any $m\in(0,1]$,
    \begin{equation}\label{lower_bound_inequality}
        c(A_1+mA_2)\geq\min_{\sigma\in S_2}\frac{1-r_{\sigma(1)}-2r_{\sigma(2)}}{3+r_{\sigma(1)}+2r_{\sigma(2)}}.
    \end{equation}
\end{lemma}
\begin{proof}
    Since $A_1$ is an L-set and $A_2$ is an R-set, we can use Lemma \ref{LR-set}. Then
    \begin{equation*}
        \begin{split}
            A_1&=\left[0,r_1+n_1\left(\frac{1-r_1}{2}\right)\right]\cup \left[1,1+r_1-n_1\left(\frac{1+r_1}{2}\right)\right]=:I_1\cup J_1,\\
            mA_2&=\left[0,mr_2-mn_2\left(\frac{1-r_2}{2}\right)\right]\cup\left[m,m+mr_2+mn_2\left(\frac{1+r_2}{2}\right)\right]=:I_2\cup J_2.
        \end{split}
    \end{equation*}
    Using Minkowski addition of the sets we get
    \begin{equation*}
        A_1+mA_2=(I_1+I_2)\cup (I_1+J_2) \cup (J_1+I_2) \cup (J_1+J_2).
    \end{equation*}
    For an arbitrary interval $I$, let $L(I)$ denote the left endpoint of $I$, and let $R(I)$ denote the right endpoint of $I$. We will now identify the candidates for the largest gap of the sumset. These candidates are
    \begin{equation*}
        \begin{split}
            g_0(m)&:=L(J_1+I_2)-R(I_1+I_2),\\
            g_1(m)&:=L(I_1+J_2)-R(I_1+I_2),\\
            g_2(m)&:=L(J_1+I_2)-R(I_1+J_2),\\
            g_3(m)&:=L(J_1+J_2)-R(J_1+I_2),\\
            g_4(m)&:=L(J_1+J_2)-R(I_1+J_2).
        \end{split}
    \end{equation*}
    We can calculate these gap candidates (purposely leaving out $g_0(m)$) explicitly as 
    \begin{equation*}
        \begin{split}
            g_1(m)&=m-r_1-mr_2-n_1\left(\frac{1-r_1}{2}\right)+mn_2\left(\frac{1-r_2}{2}\right),\\
            g_2(m)&=1-m-r_1-mr_2-n_1\left(\frac{1-r_1}{2}\right)-mn_2\left(\frac{1+r_2}{2}\right),\\
            g_3(m)&=m-r_1-mr_2+n_1\left(\frac{1+r_1}{2}\right)+mn_2\left(\frac{1-r_2}{2}\right),\\
            g_4(m)&=1-r_1-mr_2-n_1\left(\frac{1-r_2}{2}\right)-mn_2\left(\frac{1+r_2}{2}\right).
        \end{split}
    \end{equation*}

    \begin{claim}\label{biggest_gap_formula}
    The largest gap of the sumset takes the form
    \begin{equation*}
        G(A_1+mA_2)=\max\left\{g_1(m),g_2(m),\min\{g_3(m),g_4(m)\},0\right\}.
    \end{equation*}
    \end{claim}
    \begin{proof}[Proof of Claim \ref{biggest_gap_formula}]
        There are four intervals in the sumset: $I_1+I_2$, $J_1+I_2$, $I_1+J_2$, and $J_1+J_2$. Since $R(I_1+I_2)<R(I_1+J_2)$, the gap $g_0(m)$ is never a candidate for the largest gap (except when $m=1$, in which case $g_1(m)=g_0(m)$). We will show that $R(I_1+I_2)<1=L(J_1+I_2)$. To verify this, we compute
        \begin{equation*}
            \begin{split}
                R(I_1+I_2)&=r_1+mr_2+n_1\left(\frac{1-r_1}{2}\right)-mn_2\left(\frac{1-r_2}{2}\right)\\
                &\leq2r_1+r_2<1.
            \end{split}
        \end{equation*}
        Then $R(I_1+I_2)<1=L(I_2+J_1)$ which implies that the interval $I_1+I_2$ is always some positive distance to the left of the interval $I_2+J_1$ and the gap between these two intervals happens to be $g_0(m)$. The two intervals yet unaccounted for are $I_1+J_2$ and $J_1+J_2$. First, since $L(J_1+J_2)>L(I_2+J_1)$, the interval $J_1+J_2$ can never intersect the gap $g_0(m)$. Since $L(I_1+J_2)=m\in(0,1]$, the interval $I_1+J_2$ will cover part of the gap $g_0(m)$. This implies that $g_1(m)$ and $g_2(m)$ are candidates for the largest gap. The only other possibility for a gap in the sumset is the gap between the intervals $J_1+J_2$ and $I_2+J_1$, which is $g_3(m)$, in the case that $R(I_1+J_2)\leq R(I_2+J_1)$, or the gap between the intervals $J_1+J_2$ and $I_1+J_2$, which is $g_4(m)$, in the case that $R(I_1+J_2)>R(I_2+J_1)$. In other words, the other candidate for the largest gap is $\min(g_3(m),g_4(m))$. It follows that 
        \begin{equation}\label{maximum_term}
            G(A_1+mA_2)=\max\{g_1(m),g_2(m),\min\{g_3(m),g_4(m)\},0\},
        \end{equation}
        where we include the $0$ for the possibility that all the gap candidates are negative (which we will find cannot happen). This completes the proof of the claim.
    \end{proof}
    
For the next part of this lemma, we will omit the $0$ from Claim \ref{biggest_gap_formula}. If it turns out that the lower bound we achieve is a positive number (which it will), then it will follow that the $0$ in the formula for the largest gap is extraneous and we will have proved the lemma. We only need to consider two cases.
\begin{case}
    $\min\{g_3(m),g_4(m)\}=g_3(m)$.
\end{case}
In this case we have
\begin{equation*}
    G(A_1+mA_2)=\max\{g_1(m),g_2(m),g_3(m)\}\geq\max\{g_2(m),g_3(m)\}.
\end{equation*}
Letting $d_{12}(m)$ denote the diameter of the sumset, the convexity index of the sumset is
\begin{equation}\label{LR_index}
    c(A_1+mA_2)\geq\max\left\{\frac{g_2(m)}{d_{12}(m)},\frac{g_3(m)}{d_{12}(m)}\right\}=:\max\{F_2(m),F_3(m)\}.
\end{equation}
We will show that $F_2(m)$ is non-increasing with $m$, and that $F_3(m)$ is non-decreasing with $m$. It will follow that the quantity on the right of (\ref{LR_index}) is minimized when $F_2(m)=F_3(m)$. We will start with $F_2(m)$. Write
\begin{equation}\label{LR-F_2}
    F_2(m)=\frac{1-m(1+r_2)\left(1+\frac{n_2}{2}\right)-r_1-n_1\left(\frac{1-r_1}{2}\right)}{1+m(1+r_2)\left(1+\frac{n_2}{2}\right)+r_1-n_1\left(\frac{1+r_1}{2}\right)}.
\end{equation}
Then
\begin{equation*}
    F_2^{\prime}(m)\approx(1+r_2)\left(1+\frac{n_2}{2}\right)(n_1-2)\leq0.
\end{equation*}
The last inequality follows from the fact given in Lemma \ref{LR-set} that $n_1\leq\frac{2r_1}{1+r_1}\leq 2$. For $F_3(m)$, write
\begin{equation*}
    F_3(m)=\frac{m(1-r_2)\left(1+\frac{n_2}{2}\right)-r_1+n_1\left(\frac{1+r_1}{2}\right)}{1+m(1+r_2)\left(1+\frac{n_2}{2}\right)+r_1-n_1\left(\frac{1+r_1}{2}\right)}.
\end{equation*}
Then
\begin{equation*}
    \begin{split}
        F_3^{\prime}(m)\approx&(1-r_2)\left(1+\frac{n_2}{2}\right)\left(1+r_1-n_1\left(\frac{1+r_1}{2}\right)\right)\\
        &+(1+r_2)\left(1+\frac{n_2}{2}\right)\left(r_1-n_1\left(\frac{1+r_1}{2}\right)\right)\geq0.
    \end{split}
\end{equation*}
We used the fact from Lemma \ref{LR-set} that $n_1(1+r_1)\leq 2r_1$. This verifies what we needed, so we minimize now. We have that $F_2(m)=F_3(m)$ if and only if $m=(1-n_1)(2+n_2)^{-1}$. Now substitute this value of $m$ into equation (\ref{LR-F_2}) to get
\begin{equation*}
    c(A_1+mA_2)\geq\frac{1-(1+r_2)(1-n_1)\left(\frac{1+\frac{n_2}{2}}{2+n_2}\right)-r_1-n_1\left(\frac{1-r_1}{2}\right)}{1+(1+r_2)(1-n_1)\left(\frac{1+\frac{n_2}{2}}{2+n_2}\right)+r_1-n_1\left(\frac{1+r_1}{2}\right)}:=\gamma(n_1,n_2).
\end{equation*}
First we notice that
\begin{equation*}
    \frac{1+\frac{n_2}{2}}{2+n_2}=\frac{1}{2}.
\end{equation*}
Then
\begin{equation*}
\begin{split}
    \gamma(n_1,n_2)&=\frac{1-\frac{1}{2}(1+r_2)(1-n_1)-r_1-n_1\left(\frac{1-r_1}{2}\right)}{1+\frac{1}{2}(1+r_2)(1-n_1)+r_1-n_1\left(\frac{1+r_1}{2}\right)}\\
    &=\frac{1-2r_1-r_2+n_1(r_1+r_2)}{3+2r_1+r_2-n_1(2+r_1+r_2)}\\
    &\geq\frac{1-2r_1-r_2}{3+2r_1+r_2}.
\end{split}
\end{equation*}
This proves the lower bound.
\begin{case}
    $\min\{g_3(m),g_4(m)\}=g_4(m)$.
\end{case}
Then we have
\begin{equation*}
    G(A_1+mA_2)=\max\{g_1(m),g_2(m),g_4(m)\}\geq\max\{g_1(m),g_4(m)\}.
\end{equation*} 
In this case the convexity index is given by
\begin{equation}\label{LR_F1_F4}
    c(A_1+mA_2)\geq\max\{F_1(m),F_4(m)\},
\end{equation}
where we recall that $F_j(m):=\frac{g_j(m)}{d_{12}(m)}$. We will show that $F_1(m)$ is non-decreasing in $m$, and that $F_4(m)$ is non-increasing with $m$. It will follow that the right side of (\ref{LR_F1_F4}) is minimized when $F_1(m)=F_4(m)$. For $F_1(m)$, write
\begin{equation*}
    F_1(m)=\frac{m(1-r_2)\left(1+\frac{n_2}{2}\right)-r_1-n_1\left(\frac{1-r_1}{2}\right)}{1+m(1+r_2)\left(1+\frac{n_2}{2}\right)+r_1-n_1\left(\frac{1+r_1}{2}\right)}.
\end{equation*}
Then
\begin{equation*}
\begin{split}
    F_1^{\prime}(m)\approx&(1-r_2)\left(1+\frac{n_2}{2}\right)\left(1+r_1-n_1\left(\frac{1+r_1}{2}\right)\right)\\
    &+(1+r_2)\left(1+\frac{n_2}{2}\right)\left(r_1+n_1\left(\frac{1-r_1}{2}\right)\right)\geq0.
\end{split}
\end{equation*}
Hence $F_1(m)$ is non-decreasing.
For $F_4(m)$, write
\begin{equation*}
    F_4(m)=\frac{1-r_1-m\left(r_2+n_2\left(\frac{1+r_2}{2}\right)\right)-n_1\left(\frac{1-r_1}{2}\right)}{1+m(1+r_2)\left(1+\frac{n_2}{2}\right)+r_1-n_1\left(\frac{1+r_1}{2}\right)}.
\end{equation*}
Then
\begin{equation*}
\begin{split}
    F_4^{\prime}(m)\approx&-\left(r_2+n_2\left(\frac{1+r_2}{2}\right)\right)\left(1+r_1-n_1\left(\frac{1+r_1}{2}\right)\right)\\
    &-(1+r_2)\left(1+\frac{n_2}{2}\right)\left(1-r_1-n_1\left(\frac{1-r_1}{2}\right)\right)\leq0.
\end{split}
\end{equation*}
Hence $F_4(m)$ is non-increasing. Now, $F_1(m)=F_4(m)$ if and only if $m=\frac{1}{1+n_2}$. Substituting this value of $m$ into $F_1(m)$ we get
\begin{equation*}
    c(A_1+mA_2)\geq\frac{(1-r_2)\left(\frac{1+\frac{n_2}{2}}{1+n_2}\right)-r_1-n_1\left(\frac{1-r_1}{2}\right)}{1+(1+r_2)\left(\frac{1+\frac{n_2}{2}}{1+n_2}\right)+r_1-n_1\left(\frac{1+r_1}{2}\right)}=:\gamma(n_1,n_2).
\end{equation*}
The function $\gamma(n_1,n_2)$ is decreasing in the variables $n_1$ and $n_2$. To see this, assume that $r_1\geq r_2$. Then
\begin{equation*}
\begin{split}
    D_1\gamma(n_1,n_2)&\approx-\left(\frac{1+r_1}{2}\right)+\left(\frac{1+\frac{n_2}{2}}{1+n_2}\right)(r_1-r_2)\\
    &\leq^{*}-\left(\frac{1+r_1}{2}\right)+(r_1-r_2)\\
    &=\frac{1}{2}(-1+r_1-2r_2)\leq0.
\end{split}
\end{equation*}
The $\leq^{*}$ followed from the fact that the quantity $\frac{1+\frac{n_2}{2}}{1+n_2}$ is decreasing in $n_2$, and so is maximized when $n_2=0$. We can see by observing the above calculation that the above quantity is non-positive in the case that $r_1\leq r_2$. For the variable $n_2$ observe that
\begin{equation*}
    D_2\gamma(n_1,n_2)\approx-\frac{1-r_2}{2(1+n_2)^2}\left[1+r_1-n_1\left(\frac{1+r_1}{2}\right)\right]-\frac{1+r_2}{2(1+n_2)^2}\left[r_1+n_1\left(\frac{1-r_1}{2}\right)\right]\leq0.
\end{equation*}
Then, since $n_1\leq\frac{2r_1}{1+r_1}$ and $n_2\leq\frac{2r_2}{1-r_2}$ the lower bound for $\gamma(n_1,n_2)$ is 
\begin{equation}\label{LR_gamma}
    \gamma(n_1,n_2)\geq\gamma\left(\frac{2r_1}{1+r_1},\frac{2r_2}{1-r_2}\right)=\frac{1-r_1-r_2-3r_1r_2}{2+2r_1+2r_2+2r_1r_2}.
\end{equation}
Now, for at least one $j\in[2]$, $r_j\leq\frac{1}{3}$. If not, then for any $\sigma\in S_2$,
\begin{equation*}
    1-r_{\sigma(1)}-2r_{\sigma(2)}<1-\frac{1}{3}-\frac{2}{3}=0.
\end{equation*}
This is a contradiction. For each $j\in[2]$, we must have $r_j\leq \frac{1}{2}$. To see this, observe that for each $\sigma\in S_2$,
\begin{equation*}
    2r_{\sigma(2)}\leq r_{\sigma(1)}+2r_{\sigma(2)}<1.
\end{equation*}
Then $r_{\sigma(2)}\leq\frac{1}{2}$. Wlog, suppose that $r_1=\min\{r_1,r_2\}$. Then $r_1\leq\frac{1}{3}$, and therefore $-3r_1r_2\geq -r_2$. Also,
\begin{equation*}
    r_1+2r_1r_2\leq\frac{1}{2}+\frac{2}{4}=1.
\end{equation*}
Putting this all together, we find using (\ref{LR_gamma}) that
\begin{equation*}
    \gamma(n_1,n_2)\geq \frac{1-r_1-2r_2}{3+r_1+2r_2}.
\end{equation*}
This proves the bound we need, completing the proof of the lemma.
\end{proof}

\begin{lemma}\label{R-set,L-set}
    Let $A_1$ be an R-set, and $A_2$ be an L-set such that (\ref{positive_assumption}) holds. Then for any $m\in(0,1]$, (\ref{lower_bound_inequality}) holds.
\end{lemma}
\begin{proof}
    By Lemma \ref{LR-set} we can write
    \begin{equation*}
    \begin{split}
        A_1&=\left[0,r_1-n_1\left(\frac{1-r_1}{2}\right)\right]\cup\left[1,1+r_1+n_1\left(\frac{1+r_1}{2}\right)\right]=:I_1\cup J_1,\\
        mA_2&=\left[0,mr_2+mn_2\left(\frac{1-r_2}{2}\right)\right]\cup\left[m,m+mr_2-mn_2\left(\frac{1+r_2}{2}\right)\right]=:I_2\cup J_2.
    \end{split}
    \end{equation*}
    Using Minkowski addition we get
    \begin{equation*}
        A_1+mA_2=(I_1+I_2)\cup(I_1+J_2)\cup(J_1+I_2)\cup(J_1+J_2).
    \end{equation*}
    The possible gaps are $g_j$, $1\leq j\leq 4$, just as defined in the proof of Lemma \ref{L-set,R-set}. In fact, they can be written explicitly as
    \begin{equation*}
        \begin{split}
            g_1(m)&=m-r_1-mr_2+n_1\left(\frac{1-r_1}{2}\right)-mn_2\left(\frac{1-r_2}{2}\right),\\
            g_2(m)&=1-m-r_1-mr_2+n_1\left(\frac{1-r_1}{2}\right)+mn_2\left(\frac{1+r_2}{2}\right),\\
            g_3(m)&=m-r_1-mr_2-n_1\left(\frac{1+r_1}{2}\right)-mn_2\left(\frac{1-r_2}{2}\right),\\
            g_4(m)&=1-r_1-mr_2+n_1\left(\frac{1-r_1}{2}\right)+mn_2\left(\frac{1+r_2}{2}\right).
        \end{split}
    \end{equation*}

    \begin{claim}\label{biggest_gap_RL}
        The largest gap of the sumset is
        \begin{equation*}
            G(A_1+mA_2)=\max\{g_1(m),g_2(m),\min\{g_3(m),g_4(m)\},0\}.
        \end{equation*}
    \end{claim}

    \begin{proof}[Proof of Claim \ref{biggest_gap_RL}]
        This will be proved just as in Claim \ref{biggest_gap_formula}. By careful observation, the entire proof is identical except we need to verify that $R(I_1+I_2)<1$. We can see this by calculating
        \begin{equation*}
            \begin{split}
                R(I_1+I_2)&=r_1+mr_2-n_1\left(\frac{1-r_1}{2}\right)+mn_2\left(\frac{1-r_2}{2}\right)\\
                &\leq r_1+2r_2<1.
            \end{split}
        \end{equation*}
        This proves the claim.
    \end{proof}
    
    Just as before, we will continue the proof by omitting the $0$ in Claim \ref{biggest_gap_RL}. There are two cases to consider.
    \begin{case}
        $\min\{g_3(m),g_4(m)\}=g_3(m)$.
    \end{case}
    Then 
    \begin{equation}\label{RL_index}
        c(A_1+mA_2)=\frac{\max\{g_1(m),g_2(m),g_3(m)\}}{d_{12}(m)}\geq\frac{\max\{g_1(m),g_2(m)\}}{d_{12}(m)}=:\max\{F_1(m),F_2(m)\}.
    \end{equation}
    It will be shown that $F_1(m)$ is non-decreasing with $m$, and that $F_2(m)$ is non-increasing with $m$. Then it will follow that the right side of (\ref{RL_index}) is minimized when $F_1(m)=F_2(m)$. For $F_1(m)$, write
    \begin{equation*}
        F_1(m)=\frac{m(1-r_2)\left(1-\frac{n_2}{2}\right)-r_1+n_1\left(\frac{1-r_1}{2}\right)}{1+m(1+r_2)\left(1-\frac{n_2}{2}\right)+r_1+n_1\left(\frac{1+r_1}{2}\right)}.
    \end{equation*}
    Then
    \begin{equation*}
        \begin{split}
         F_1^{\prime}(m)\approx&(1-r_2)\left(1-\frac{n_2}{2}\right)\left(1+r_1+n_1\left(\frac{1+r_1}{2}\right)\right)\\
        &+(1+r_2)\left(1-\frac{n_2}{2}\right)\left(r_1-n_1\left(\frac{1-r_1}{2}\right)\right)\geq0,
        \end{split}
    \end{equation*}
    which verifies that $F_1(m)$ is non-decreasing. For $F_2(m)$, write
    \begin{equation*}
        F_2(m)=\frac{1-m(1+r_2)\left(1-\frac{n_2}{2}\right)-r_1+n_1\left(\frac{1-r_1}{2}\right)}{1+m(1+r_2)\left(1-\frac{n_2}{2}\right)+r_1+n_1\left(\frac{1+r_1}{2}\right)}.
    \end{equation*}
    Then
    \begin{equation*}
        \begin{split}
            F_2^{\prime}(m)\approx&-(1+r_2)\left(1-\frac{n_2}{2}\right)\left(1+r_1+n_1\left(\frac{1+r_1}{2}\right)\right)\\
            &-(1+r_2)\left(1-\frac{n_2}{2}\right)\left(1-r_1+n_1\left(\frac{1-r_1}{2}\right)\right)\leq0,
        \end{split}
    \end{equation*}
    which verifies that $F_2(m)$ is non-increasing. Now, $F_1(m)=F_2(m)$ if and only if $m=\frac{1}{2-n_2}$. Then
    \begin{equation*}
        \begin{split}
             \max\{F_1(m),F_2(m)\}&\geq \frac{1-(1+r_2)\left(\frac{1-\frac{n_2}{2}}{2-n_2}\right)-r_1+n_1\left(\frac{1-r_1}{2}\right)}{1+(1+r_2)\left(\frac{1-\frac{n_2}{2}}{2-n_2}\right)+r_1+n_1\left(\frac{1+r_1}{2}\right)}\\
             &=\frac{1-\frac{1}{2}(1+r_2)-r_1+n_1\left(\frac{1-r_1}{2}\right)}{1+\frac{1}{2}(1+r_2)+r_1+n_1\left(\frac{1+r_1}{2}\right)}\\
             &=\frac{1-2r_1-r_2+n_1(1-r_1)}{3+2r_1+r_2+n_1(1+r_1)}=:\gamma(n_1).
        \end{split}
    \end{equation*}
    Now, $\gamma^{\prime}(n_1)\approx2+2r_2>0$, so that $\gamma$ is increasing in $n_1$. Then
    \begin{equation*}
        \gamma(n_1)\geq\gamma(0)=\frac{1-2r_1-r_2}{3+2r_1+r_2}.
    \end{equation*}
    This proves the lower bound.
    \begin{case}
        $\min\{g_3(m),g_4(m)\}=g_4(m)$.
    \end{case}
    Then
    \begin{equation*}
        c(A_1+mA_2)=\frac{\max\{g_1(m),g_2(m),g_4(m)\}}{d_{12}(m)}\geq\frac{\max\{g_1(m),g_4(m)\}}{d_{12}(m)}=:\max\{F_1(m),F_4(m)\}.
    \end{equation*}
    We already know that $F_1(m)$ is non-decreasing. We will show that $F_4(m)$ is non-increasing. Then we minimize by finding where $F_1(m)=F_4(m)$. For $F_4(m)$ write
    \begin{equation*}
        F_4(m)=\frac{1-r_1-m\left(r_2-n_2\left(\frac{1+r_2}{2}\right)\right)+n_1\left(\frac{1-r_1}{2}\right)}{1+m(1+r_2)\left(1-\frac{n_2}{2}\right)+r_1+n_1\left(\frac{1+r_1}{2}\right)}.
    \end{equation*}
    Then 
    \begin{equation*}
        \begin{split}
             F_4^{\prime}(m)\approx& -\left(r_2-n_2\left(\frac{1+r_2}{2}\right)\right)\left(1+r_1+n_1\left(\frac{1+r_1}{2}\right)\right)\\
             &-(1+r_2)\left(1-\frac{n_2}{2}\right)\left(1-r_1+n_1\left(\frac{1-r_1}{2}\right)\right)\leq0,
        \end{split}
    \end{equation*}
    which verifies that $F_4(m)$ is non-increasing. Now, $F_1(m)=F_4(m)$ if and only if $m=\frac{1}{1-n_2}$. Then substituting $m=\frac{1}{1-n_2}$ into $F_1(m)$, we get
    \begin{equation*}
        \max\{F_1(m),F_4(m)\}\geq \frac{(1-r_2)\left(\frac{1-\frac{n_2}{2}}{1-n_2}\right)-r_1+n_1\left(\frac{1-r_1}{2}\right)}{1+(1+r_2)\left(\frac{1-\frac{n_2}{2}}{1-n_2}\right)+r_1+n_1\left(\frac{1+r_1}{2}\right)}=:\gamma(n_1,n_2).
    \end{equation*}
    Now, suppose that $r_2<r_1$. Since $\left(\frac{1-\frac{n_2}{2}}{1-n_2}\right)$ is increasing in $n_2$ and $n_2\leq \frac{2r_2}{1+r_2}$, we have
    \begin{equation*}
        \left(\frac{1-\frac{n_2}{2}}{1-n_2}\right)\leq\frac{1}{1-r_2}.
    \end{equation*}
    It follows that 
    \begin{equation*}
    \begin{split}
        D_1\gamma(n_1,n_2)&\approx \frac{1}{2}+\frac{r_1}{2}+\left(\frac{1-\frac{n_2}{2}}{1-n_2}\right)(r_2-r_1)\\
        &\geq \frac{1}{2}+\frac{r_1}{2}+\frac{r_2-r_1}{1-r_2}\\
        &=\frac{1-r_1+r_2-r_1r_2}{2(1-r_2)}\geq0.
    \end{split}
    \end{equation*}
    In the case the $r_2\geq r_1$, one can verify by inspection that the above quantity is non-negative. It follows that $\gamma(n_1,n_2)\geq \gamma(0,n_2)$. We also have that
    \begin{equation*}
        D_2\gamma(0,n_2)\approx \frac{1+2r_1-r_2}{2(1-n_2)^{2}}\geq0.
    \end{equation*}
    Then
    \begin{equation*}
        \gamma(n_1,n_2)\geq\gamma(0,0)=\frac{1-r_1-r_2}{2+r_1+r_2}\geq\frac{1-2r_1-r_2}{3+2r_1+r_2},
    \end{equation*}
    which completes the proof of the lemma.
\end{proof}
\begin{lemma}\label{L-set,L-set}
    Let $A_1$ and $A_2$ be L-sets such that (\ref{positive_assumption}) holds. Then for any $m\in(0,1]$, (\ref{lower_bound_inequality}) holds.
\end{lemma}
\begin{proof}
    Using Lemma \ref{LR-set} we may write
    \begin{equation*}
        \begin{split}
            A_1&=\left[0,r_1+n_1\left(\frac{1-r_1}{2}\right)\right]\cup\left[1,1+r_1-n_1\left(\frac{1+r_1}{2}\right)\right]=:I_1\cup J_1,\\
            mA_2&=\left[0,mr_2+mn_2\left(\frac{1-r_2}{2}\right)\right]\cup\left[m,m+mr_2-mn_2\left(\frac{1+r_2}{2}\right)\right]=:I_2\cup J_2.
        \end{split}
    \end{equation*}
    Adding the sets together we get
    \begin{equation*}
        A_1+mA_2=(I_1+I_2)\cup(I_1+J_2)\cup(J_1+I_2)\cup(J_1+J_2).
    \end{equation*}
    The gaps $g_j$, $1\leq j\leq 4$ are defined the same as in Lemma \ref{L-set,R-set}. These gaps can be written explicitly as
    \begin{equation*}
        \begin{split}
            g_1(m)&=m-r_1-mr_2-n_1\left(\frac{1-r_1}{2}\right)-mn_2\left(\frac{1-r_2}{2}\right),\\
            g_2(m)&=1-m-r_1-mr_2-n_1\left(\frac{1-r_1}{2}\right)+mn_2\left(\frac{1+r_2}{2}\right),\\
            g_3(m)&=m-r_1-mr_2+n_1\left(\frac{1+r_1}{2}\right)-mn_2\left(\frac{1-r_2}{2}\right),\\
            g_4(m)&=1-r_1-mr_2-n_1\left(\frac{1-r_1}{2}\right)+mn_2\left(\frac{1+r_2}{2}\right).
        \end{split}
    \end{equation*}

    \begin{claim}\label{biggest_gap_LL}
        The largest gap of the sumset is given by
        \begin{equation*}
            G(A_1+mA_2)=\max\{g_1(m),g_2(m),\min\{g_3(m),g_4(m)\},0\}.
        \end{equation*}
    \end{claim}

    \begin{proof}[Proof of Claim \ref{biggest_gap_LL}]
        All we need to do is to check the same conditions that were verified in Claim \ref{biggest_gap_RL}. We have
        \begin{equation*}
            \begin{split}
                R(I_1+I_2)&=r_1+mr_2+n_1\left(\frac{1-r_2}{2}\right)+mn_2\left(\frac{1-r_2}{2}\right)\\
                &\leq r_1+r_2+\frac{r_1(1-r_1)}{1+r_1}+\frac{r_2(1-r_2)}{1+r_2}\\
                &=\frac{2r_1}{1+r_1}+\frac{2r_2}{1+r_2}\\
                &=\frac{2r_1+2r_2+4r_1r_2}{1+r_1+r_2+r_1r_2}=:P(r_1,r_2).
            \end{split}
        \end{equation*}
        Now, since $\min(r_1,r_2)\leq\frac{1}{3}$ we have
        \begin{equation*}
            \begin{split}
                P(r_1,r_2)-1&\approx (2r_1+2r_2+4r_1r_2)-(1+r_1+r_2+r_1r_2)\\
                &=r_1+r_2+3r_1r_2-1\\
                &\leq r_1+r_2+\max(r_1,r_2)-1<0.
            \end{split}
        \end{equation*}
        We have verified that $R(I_1+I_2)<1$, which proves the claim.
    \end{proof}
    
    Omitting the $0$ in Claim \ref{biggest_gap_LL}, we only need the following two cases.
    \begin{case}
        $\min\{g_3(m),g_4(m)\}=g_3(m)$.
    \end{case}
    Then
    \begin{equation*}
        c(A_1+mA_2)=\frac{\max\{g_1(m),g_2(m),g_3(m)\}}{d_{12}(m)}\geq\frac{\max\{g_2(m),g_3(m)\}}{d_{12}(m)}=:\max\{F_2(m),F_3(m)\}.
    \end{equation*}
    It will be shown that $F_2(m)$ is non-increasing, and that $F_3(m)$ is non-decreasing. Then to minimize, we will find where $F_2(m)=F_3(m)$. Now,
    \begin{equation*}
        F_2(m)=\frac{1-m(1+r_2)\left(1-\frac{n_2}{2}\right)-r_1-n_1\left(\frac{1-r_1}{2}\right)}{1+m(1+r_2)\left(1-\frac{n_2}{2}\right)+r_1-n_1\left(\frac{1+r_1}{2}\right)}.
    \end{equation*}
    Then
    \begin{equation*}
        F_2^{\prime}(m)\approx -(1+r_2)\left(1-\frac{n_2}{2}\right)(2-n_1)\leq0,
    \end{equation*}
    which verifies that $F_2(m)$ is non-increasing. For $F_3(m)$ we have
    \begin{equation*}
        F_3(m)=\frac{m(1-r_2)\left(1-\frac{n_2}{2}\right)-r_1+n_1\left(\frac{1+r_1}{2}\right)}{1+m(1+r_2)\left(1-\frac{n_2}{2}\right)+r_1-n_1\left(\frac{1+r_1}{2}\right)}.
    \end{equation*}
    Then
    \begin{equation*}
        \begin{split}
            F_3^{\prime}(m)\approx& (1-r_2)\left(1-\frac{n_2}{2}\right)\left(1+r_1-n_1\left(\frac{1+r_1}{2}\right)\right)\\
            &+(1+r_2)\left(1-\frac{n_2}{2}\right)\left(r_1-n_1\left(\frac{1+r_1}{2}\right)\right)\geq0,
        \end{split}
    \end{equation*}
    which verifies that $F_3(m)$ is non-decreasing. Now, $F_2(m)=F_3(m)$ if and only if $m=\frac{1-n_1}{2-n_2}$. Then substituting $m=\frac{1-n_1}{2-n_2}$ into $F_3(m)$ we get
    \begin{equation*}
        \begin{split}
            \max\{F_2(m),F_3(m)\}&\geq\frac{(1-r_2)(1-n_1)\left(\frac{1-\frac{n_2}{2}}{2-n_2}\right)-r_1+n_1\left(\frac{1+r_1}{2}\right)}{1+(1+r_2)(1-n_1)\left(\frac{1-\frac{n_2}{2}}{2-n_2}\right)+r_1-n_1\left(\frac{1+r_1}{2}\right)}\\
            &=\frac{\frac{1}{2}(1-r_2)(1-n_1)-r_1+n_1\left(\frac{1+r_1}{2}\right)}{1+\frac{1}{2}(1+r_2)(1-n_1)+r_1-n_1\left(\frac{1+r_1}{2}\right)}\\
            &=\frac{1-2r_1-r_2+n_1(r_1+r_2)}{3+2r_1+r_2-n_1(2+r_1+r_2)}\\
            &\geq\frac{1-2r_1-r_2}{3+2r_1+r_2}.
        \end{split}
    \end{equation*}
    This proves the lower bound we need.
    \begin{case}
        $\min\{g_3(m),g_4(m)\}=g_4(m)$.
    \end{case}
    Then 
    \begin{equation*}
        c(A_1+mA_2)=\max\{F_1(m),F_2(m),F_4(m)\}\geq\max\{F_1(m),F_4(m)\}.
    \end{equation*}
    We will show that $F_1(m)$ is non-decreasing, and that $F_4(m)$ is non-increasing. For $F_1(m)$, write
    \begin{equation*}
        F_1(m)=\frac{m(1-r_2)\left(1-\frac{n_2}{2}\right)-r_1-n_1\left(\frac{1-r_1}{2}\right)}{1+m(1+r_2)\left(1-\frac{n_2}{2}\right)+r_1-n_1\left(\frac{1+r_1}{2}\right)}.
    \end{equation*}
    Then
    \begin{equation*}
        \begin{split}
            F_1^{\prime}(m)\approx& (1-r_2)\left(1-\frac{n_2}{2}\right)\left(1+r_1-n_1\left(\frac{1+r_1}{2}\right)\right)\\
            &+(1+r_2)\left(1-\frac{n_2}{2}\right)\left(r_1+n_1\left(\frac{1-r_1}{2}\right)\right)\geq0,
        \end{split}
    \end{equation*}
    which verifies that $F_1(m)$ is non-decreasing. For $F_4(m)$, write 
    \begin{equation*}
        F_4(m)=\frac{1-r_1-m\left(r_2-n_2\left(\frac{1+r_2}{2}\right)\right)-n_1\left(\frac{1-r_1}{2}\right)}{1+m(1+r_2)\left(1-\frac{n_2}{2}\right)+r_1-n_1\left(\frac{1+r_1}{2}\right)}.
    \end{equation*}
    Then
    \begin{equation*}
        \begin{split}
            F_4^{\prime}(m)\approx& -\left(r_2-n_2\left(\frac{1+r_2}{2}\right)\right)\left(1+r_1-n_1\left(\frac{1+r_1}{2}\right)\right)\\
             &-(1+r_2)\left(1-\frac{n_2}{2}\right)\left(1-r_1-n_1\left(\frac{1-r_1}{2}\right)\right)\leq0,
        \end{split}
    \end{equation*}
    which verifies that $F_4(m)$ is non-increasing. Now, $F_1(m)=F_4(m)$ if and only if $m=\frac{1}{1-n_2}$. Then substituting $m=\frac{1}{1-n_2}$ into $F_1(m)$ we get
    \begin{equation*}
        \max\{F_1(m),F_4(m)\}\geq\frac{(1-r_2)\left(\frac{1-\frac{n_2}{2}}{1-n_2}\right)-r_1-n_1\left(\frac{1-r_1}{2}\right)}{1+(1+r_2)\left(\frac{1-\frac{n_2}{2}}{1-n_2}\right)+r_1-n_1\left(\frac{1+r_1}{2}\right)}=:\gamma(n_1,n_2).
    \end{equation*}
    Suppose that $r_1>r_2$. Since $\left(\frac{1-\frac{n_2}{2}}{1-n_2}\right)$ is increasing, and $n_2\leq\frac{2r_2}{1+r_2}$ we have 
    \begin{equation*}
        \left(\frac{1-\frac{n_2}{2}}{1-n_2}\right)\leq \frac{1}{1-r_2}.
    \end{equation*}
    Then
    \begin{equation*}
    \begin{split}
        D_1\gamma(n_1,n_2)&\approx -\frac{1}{2}-\frac{r_1}{2}+\left(\frac{1-\frac{n_2}{2}}{1-n_2}\right)(r_1-r_2)\\
        &\leq -\frac{1}{2}-\frac{r_1}{2}+\left(\frac{r_1-r_2}{1-r_2}\right)\\
        &=\frac{-1+r_1-r_2+r_1r_2}{2(1-r_2)}\leq0.
    \end{split}
    \end{equation*}
    If $r_1\leq r_2$, then it can be seen that the above quantity is not positive. Then, since $n_1\leq\frac{2r_1}{1+r_1}$ we have
    \begin{equation*}
        \gamma(n_1,n_2)\geq\gamma\left(\frac{2r_1}{1+r_1},n_2\right)=\frac{(1-r_2)\left(\frac{1-\frac{n_2}{2}}{1-n_2}\right)-\frac{2r_1}{1+r_1}}{1+(1+r_2)\left(\frac{1-\frac{n_2}{2}}{1-n_2}\right)}.
    \end{equation*}
    Now,
    \begin{equation*}
        D_2\gamma\left(\frac{2r_1}{1+r_1},n_2\right)\approx \frac{1}{2(1-n_2)^2}\left[1-r_2+\frac{2r_1(1+r_2)}{1+r_1}\right]\geq0.
    \end{equation*}
    Then
    \begin{equation*}
        \gamma(n_1,n_2)\geq \gamma\left(\frac{2r_1}{1+r_1},0\right)=\frac{1-r_1-r_2-r_1r_2}{2+2r_1+r_2+r_1r_2}\geq\frac{1-2r_1-r_2}{3+2r_1+r_2}.
    \end{equation*}
    This proves the lemma.
\end{proof}
\begin{lemma}\label{R-set,R-set}
     Let $A_1$ and $A_2$ be R-sets such that (\ref{positive_assumption}) holds. Then for any $m\in(0,1]$, (\ref{lower_bound_inequality}) holds.
\end{lemma}
\begin{proof}
    Using Lemma \ref{LR-set}, we can write
    \begin{equation*}
        \begin{split}
            A_1&=\left[0,r_1-n_1\left(\frac{1-r_1}{2}\right)\right]\cup\left[1,1+r_1+n_1\left(\frac{1+r_1}{2}\right)\right]=:I_1\cup J_1,\\
            mA_2&=\left[0,mr_2-mn_2\left(\frac{1-r_2}{2}\right)\right]\cup\left[m,m+mr_2+mn_2\left(\frac{1+r_2}{2}\right)\right]=:I_2\cup J_2.
        \end{split}
    \end{equation*}
    Using Minkowski addition we get
    \begin{equation*}
        A_1+mA_2=(I_1+I_2)\cup(I_1+J_2)\cup(J_1+I_2)\cup(J_1+J_2).
    \end{equation*}
    The gaps $g_j$, $1\leq j\leq 4$ are defined as in Lemma \ref{L-set,R-set}. Written explicitly the gaps are
    \begin{equation*}
        \begin{split}
             g_1(m)&=m-r_1-mr_2+n_1\left(\frac{1-r_1}{2}\right)+mn_2\left(\frac{1-r_2}{2}\right),\\
             g_2(m)&=1-m-r_1-mr_2+n_1\left(\frac{1-r_1}{2}\right)-mn_2\left(\frac{1+r_2}{2}\right),\\
             g_3(m)&=m-r_1-mr_2-n_1\left(\frac{1+r_1}{2}\right)+mn_2\left(\frac{1-r_2}{2}\right),\\
             g_4(m)&=1-r_1-mr_2+n_1\left(\frac{1-r_1}{2}\right)-mn_2\left(\frac{1+r_2}{2}\right).
        \end{split}
    \end{equation*}

    \begin{claim}\label{biggest_gap_RR}
        The biggest gap of the sumset is
        \begin{equation*}
            G(A_1+mA_2)=\max\{g_1(m),g_2(m),\min\{g_3(m),g_4(m)\},0\}.
        \end{equation*}
    \end{claim}

    \begin{proof}[Proof of Claim \ref{biggest_gap_RR}]
        Following the same routine as the previous claims, we calculate
        \begin{equation*}
            \begin{split}
                R(I_1+I_2)&=r_1+mr_2-n_1\left(\frac{1-r_1}{2}\right)-mn_2\left(\frac{1-r_2}{2}\right)\\
                &\leq r_1+r_2<1.
            \end{split}
        \end{equation*}
        We already know that $g_0(m)$ cannot be a candidate for largest gap, so we have proved the claim.
    \end{proof}
    
    Omitting the $0$ in the above claim we consider two cases.
    \begin{case}
        $\min\{g_3(m),g_4(m)\}=g_3(m)$.
    \end{case}
    Then
    \begin{equation*}
        c(A_1+mA_2)=\max\{F_1(m),F_2(m),F_3(m)\}\geq\max\{F_1(m),F_2(m)\}.
    \end{equation*}. 
    We will show that $F_1(m)$ is non-decreasing, and that $F_2(m)$ is non-increasing. For $F_1(m)$, write
    \begin{equation*}
        F_1(m)=\frac{m(1-r_2)\left(1+\frac{n_2}{2}\right)-r_1+n_1\left(\frac{1-r_1}{2}\right)}{1+m(1+r_2)\left(1+\frac{n_2}{2}\right)+r_1+n_1\left(\frac{1+r_1}{2}\right)}.
    \end{equation*}
    Then
    \begin{equation*}
        \begin{split}
            F_1^{\prime}(m)\approx &(1-r_2)\left(1+\frac{n_2}{2}\right)\left(1+r_1+n_1\left(\frac{1+r_1}{2}\right)\right)\\
            &+(1+r_2)\left(1+\frac{n_2}{2}\right)\left(r_1-n_1\left(\frac{1-r_1}{2}\right)\right)\geq0,
        \end{split}
    \end{equation*}
    which verifies that $F_1(m)$ is non-decreasing. For $F_2(m)$, write
    \begin{equation*}
        F_2(m)=\frac{1-m(1+r_2)\left(1+\frac{n_2}{2}\right)-r_1+n_1\left(\frac{1-r_1}{2}\right)}{1+m(1+r_2)\left(1+\frac{n_2}{2}\right)+r_1+n_1\left(\frac{1+r_1}{2}\right)}.
    \end{equation*}
    Then
    \begin{equation*}
        F_2^{\prime}(m)\approx -(1+r_2)\left(1+\frac{n_2}{2}\right)(2+n_1)\leq0,
    \end{equation*}
    which verifies that $F_2(m)$ is non-increasing. Now, $F_1(m)=F_2(m)$ if and only if $m=\frac{1}{2+n_2}$. Then substituting $m=\frac{1}{2+n_2}$ into $F_1(m)$ we have
    \begin{equation*}
        \begin{split}
            \max\{F_1(m),F_2(m)\}&\geq\frac{(1-r_2)\left(\frac{1+\frac{n_2}{2}}{2+n_2}\right)-r_1+n_1\left(\frac{1-r_1}{2}\right)}{1+(1+r_2)\left(\frac{1+\frac{n_2}{2}}{2+n_2}\right)+r_1+n_1\left(\frac{1+r_1}{2}\right)}\\
            &=\frac{1-2r_1-r_2+n_1(1-r_1)}{3+2r_1+r_2+n_1(1+r_1)}=:\gamma(n_1).
        \end{split}
    \end{equation*}
    We have
    \begin{equation*}
        \gamma^{\prime}(n_1)\approx 2(1+r_2)\geq0.
    \end{equation*}
    It follows that 
    \begin{equation*}
        \gamma(n_1)\geq\gamma(0)=\frac{1-2r_1-r_2}{3+2r_1+r_2}.
    \end{equation*}
    This proves the lower bound.
    \begin{case}
        $\min\{g_3(m),g_4(m)\}=g_4(m)$.
    \end{case}
    Then $m\geq\frac{1+n_1}{1+n_2}\geq\frac{1}{2+n_2}$, where we used the inequality $g_4(m)\leq g_3(m)$. Since $F_1(m)$ is non-decreasing, we use the previous case to get
    \begin{equation*}
        c(A_1+mA_2)=\max\{F_1(m),F_2(m),F_4(m)\}\geq F_1(m)\geq F_1\left(\frac{1}{2+n_2}\right)\geq \frac{1-2r_1-r_2}{3+2r_1+r_2}.
    \end{equation*}
    This completes the proof.
\end{proof}
Now we can prove the main result.
\begin{proof}[Proof of Theorem \ref{Lower_bound}]
    Let $A_1$ and $A_2$ be non-empty, compact sets in $\mathbb{R}$. Suppose that 
    \begin{equation*}
        \min_{\sigma\in S_2}(1-r_{\sigma(1)}-2r_{\sigma(2)})>0.
    \end{equation*}
    By using Lemma \ref{filling_gaps} we may assume that $A_1$ and $A_2$ are each homothetic to either an L-set, or an R-set (we will allow for the possibility that $n_1$ or $n_2$ is equal to $0$, so we also cover the possibility that either set is a balanced set). So, using translation invariance so $\min(A_i)=0$, we may assume that $A_1$ and $A_2$ have the form
    \begin{equation*}
        \begin{split}
            A_1&=[0,m_1\alpha_1]\cup[m_1,m_1+m_1\alpha_1^{\prime}],\\
            A_2&=[0,m_2\alpha_2]\cup[m_2,m_2+m_2\alpha_2^{\prime}].
        \end{split}
    \end{equation*}
    Without loss of generality, we may assume that $m_2\leq m_1$. By using invariance under affine transformations we may further assume, by replacing each set $A_i$ with $m_1^{-1}A_i$ if needed, that $A_1$ and $A_2$ have the form
    \begin{equation*}
    \begin{split}
        A_1&=[0,\alpha_1]\cup[1,1+\alpha_1^{\prime}],\\
        A_2&=[0,m\alpha_2]\cup[m,m+m\alpha_2^{\prime}],
    \end{split}
    \end{equation*}
    where $m\in[0,1]$. This is exactly the situation of Lemmas \ref{L-set,R-set},\ref{R-set,L-set},\ref{L-set,L-set}, and \ref{R-set,R-set}. This verifies the lower bound for arbitrary compact sets in $\mathbb{R}$. Now we must show that this is the best lower bound. Let $c_1,c_2\in[0,1]$. Choose $A_{\sigma(1)}=[0,r(c_{\sigma(1)})]\cup[1,1+r(c_{\sigma(1)})]$, and $A_{\sigma(2)}=[0,2r(c_{\sigma(2)})]\cup[2,2+2r(c_{\sigma(2)})]$. Then
    \begin{equation*}
    \begin{split}
        A_{\sigma(1)}+A_{\sigma(2)}=&[0,r(c_{\sigma(1)})+2r(c_{\sigma(2)})]\cup[1,1+r(c_{\sigma(1)})+2r(c_{\sigma(2)})]\\
        &\cup[2,2+r(c_{\sigma(1)})+2r(c_{\sigma(2)})]\cup[3,3+r(c_{\sigma(1)})+2r(c_{\sigma(2)})].
    \end{split}
    \end{equation*}
    If $\min_{\sigma\in S_2}(1-r_{\sigma(1)}-2r_{\sigma(2)})\leq0$, then we see that $c(A_1+A_2)=0$ as long as we choose the appropriate $\sigma\in S_2$ which achieves the minimum. If $\min_{\sigma\in S_2}(1-r_{\sigma(1)}-2r_{\sigma(2)})>0$, then choose the $\sigma\in S_2$ which minimizes, and we see that 
    \begin{equation*}
        c(A_1+A_2)=\frac{1-r(c_{\sigma(1)})-2r(c_{\sigma(2)})}{3+r(c_{\sigma(1)})+2r(c_{\sigma(2)})}.
    \end{equation*}
    In any case, for fixed convexity indices $c_1$, $c_2$, we have found compact sets $A_1,A_2$ such that $c(A_j)=c_j$ and which achieve the lower bound.
    \end{proof}
\subsection{Compact Cartesian product sets}\label{cartesian_bound}
As a corollary to Theorem \ref{Lower_bound} we can show that the convexity index of Minkowski sums of Cartesian product sets in $\mathbb{R}^n$ have a certain non-trivial lower bound.
\begin{theorem}
    Let $n\geq2$ be an integer. For $j\in[2]$ and $i\in[n]$, let $A_{ji}$ be non-empty, compact sets in $\mathbb{R}$. For $j\in[2]$ define
    \begin{equation}\label{cartesian_set}
        A_{j}:=A_{j1}\times A_{j2}\times\dots\times A_{jn}\in\mathbb{R}^n.
    \end{equation}
    Then
    \begin{equation}\label{Cartesian_product_bound}
    c(A_1+A_2)\geq\max\left(0,\max_{i\in[n]}\min_{\sigma\in S_2}\frac{1-r_{\sigma(1)i}-2r_{\sigma(2)i}}{3+r_{\sigma(1)i}+2r_{\sigma(2)i}}\right).
    \end{equation}
    Moreover, this is the best lower bound in the same sense that if $c_{ji}\in[0,1]$ for $j\in[2]$ and $i\in[n]$, then
    \begin{equation*}
        \min_{\substack{A_{ji}\in\mathcal{K}_1\\c(A_{ji})=c_{ji}\\j\in[2],i\in[n]}}c\left(A_1+A_2\right)=\max\left(0,\max_{i\in[n]}\min_{\sigma\in S_2}\frac{1-r(c_{\sigma(1)i})-2r(c_{\sigma(2)i})}{3+r(c_{\sigma(1)i})+2r(c_{\sigma(2)i}})\right),
    \end{equation*}
    where $A_1$ and $A_2$ are defined as in (\ref{cartesian_set}).
\end{theorem}
\begin{proof}
    For $\lambda\geq0$, we have
    \begin{equation*}
    \begin{split}
        (A_1+A_2)&+\lambda\textup{conv}(A_1+A_2)\\
        &=(A_{11}+A_{21})\times\dots\times(A_{1n}+A_{2n})+\lambda(\textup{conv}(A_{11}+A_{21})\times\dots\times\textup{conv}(A_{1n}+A_{2n}))\\
        &=[(A_{11}+A_{21})+\lambda\textup{conv}(A_{11}+A_{21})]\times\dots\times[(A_{1n}+A_{2n})+\lambda\textup{conv}(A_{1n}+A_{2n})].
    \end{split}
    \end{equation*}
    The above set is convex if and only if the component sets are each convex. It follows that
    \begin{equation}\label{cartesian_product_index}
        c(A_1+A_2)=\max\{c(A_{11}+A_{21}),\dots,c(A_{1n}+A_{2n})\}.
    \end{equation}
    By Theorem \ref{Lower_bound}, we have for each $i\in[n]$ that
    \begin{equation}\label{lb_1-dimension}
        c(A_{1i}+A_{2i})\geq\max\left(0,\min_{\sigma\in S_2}\frac{1-r_{\sigma(1)i}-2r_{\sigma(2)i}}{3+r_{\sigma(1)i}+2r_{\sigma(2)i}}\right).
    \end{equation}
    Putting (\ref{cartesian_product_index}) and (\ref{lb_1-dimension}) together, we have the necessary lower bound. To see that the lower bound is the best lower bound, just use the fact that (\ref{lb_1-dimension}) can be made equality according to the statement of Theorem \ref{Lower_bound},  then the maximum (\ref{cartesian_product_index}) must be equal to the bound (\ref{Cartesian_product_bound}). This completes the proof.
\end{proof}

\subsection{Open problems}\label{open_problems}
The problem of finding the best lower bound for the convexity index of Minkowski sums of arbitrary compact sets $A_1$ and $A_2$ in $\mathbb{R}^n$ when $n\geq2$ is interesting.
\begin{problem}\label{high_dim_bound}
    Find the function $F_n:[0,n]^{2}\rightarrow\mathbb{R}$ which solves the minimization problem
    \begin{equation*}
        \min_{\substack{A\in \mathcal{K}_{n}(2)\\ c(A_j)=c_j\\ j\in[2]}}c(A_1+A_2)=F_n(c_1,c_2).
    \end{equation*}
\end{problem}

It is natural to try to extend Theorem \ref{Lower_bound} to sums with $k\geq2$ compact sets in $\mathbb{R}$.
\begin{problem}\label{lower_bound_large_sum}
    Find the function $F_1:[0,1]^{k}\rightarrow\mathbb{R}$ which solves the minimization problem
    \begin{equation*}
        \min_{\substack{A\in\mathcal{K}_{1}(k)\\c(A_j)=c_j\\j\in[k]}}c(A_1+\dots+A_k)=F_1(c_1,\dots,c_k).
    \end{equation*}
\end{problem}

The next proposition gives a possible candidate for the function $F_1(c_1,\dots,c_k)$ in Problem \ref{lower_bound_large_sum}.
\begin{proposition}\label{lb_candidate}
    The solution $F_1$ to the minimization problem \ref{lower_bound_large_sum} satisfies
    \begin{equation*}
        F_1(c_1,\dots,c_k)\leq\max\left(0, \min_{\sigma\in S_k}\frac{1-\sum_{j=1}^{k}2^{j-1}r(c_{\sigma(j)})}{2^k-1+\sum_{j=1}^{k}2^{j-1}r(c_{\sigma(j)})}\right).
    \end{equation*}
\end{proposition}
\begin{proof}
    Choose $\sigma_{*}\in S_k$ so that
    \begin{equation*}
       \frac{1-\sum_{j=1}^{k}2^{j-1}r(c_{\sigma_{*}(j)})}{2^k-1+\sum_{j=1}^{k}2^{j-1}r(c_{\sigma_{*}(j)})}=\min_{\sigma\in S_k}\frac{1-\sum_{j=1}^{k}2^{j-1}r(c_{\sigma(j)})}{2^k-1+\sum_{j=1}^{k}2^{j-1}r(c_{\sigma(j)})}.
    \end{equation*}
    For each $j\in[k]$ define $A_{\sigma_{*}(j)}:=[0,2^{j-1}r(c_{\sigma_{*}(j)})]\cup[2^{j-1},2^{j-1}+2^{j-1}r(c_{\sigma_{*}(j)})]$. Then
    \begin{equation*}
        \sum_{j=1}^{k}A_{\sigma_{*}(j)}=\bigcup_{j=0}^{2^k-1}\left[j,j+\sum_{j=1}^{k}2^{j-1}r(c_{\sigma_{*}(j)})\right].
    \end{equation*}
    Since all the gaps in this sumset are the same length, the largest gap is
    \begin{equation*}
        G=\max\left(0,1-\sum_{j=1}^{k}2^{j-1}r(c_{\sigma_{*}(j)})\right).
    \end{equation*}
    The diameter of the sumset is
    \begin{equation*}
        d=2^{k}-1+\sum_{j=1}^{k}2^{j-1}r(c_{\sigma_{*}(j)}).
    \end{equation*}
    Since the convexity index is $c(A_1+\dots+A_k)=\frac{G}{d}$, the result follows immediately. 
\end{proof}

As additional motivation for solving Problem \ref{lower_bound_large_sum}, we will show how a solution will allow for an explicit formula for what is called the \textit{induced} measure of non-convexity $c^{*}$.
\begin{definition}
    For some integer $n\geq1$, let $x\in[0,n]$, and define
    \begin{equation*}
        N_x:=\inf\{k\in\mathbb{Z}:\textup{There exists } A\in \mathcal{K}_n(k) \textup{ such that } c(A_j)=x\textup{ } \forall j\in[k]\textup{ and } A_1+\dots +A_k \textup{ is convex}.\}.
    \end{equation*}
    The measure of non-convexity induced by the Schneider non-convexity index is defined by
    \begin{equation*}
        c^{*}(A):=N_{c(A)}-1.
    \end{equation*}
    If $A$ is the empty set, then set $c^{*}(A)=0$.
\end{definition}

\begin{remark}
It is noted here that the motivation for the definition of $c^{*}$ comes from the Shapley-Folkman Theorem \cite{emerson1,starr1,starr2}, which says that the $k$th set average $A(k):=(1/k)\sum_{i=1}^{k}A$ converges to $\textup{conv}(A)$ in Hausdorff distance and volume deficit. We are interested in understanding specifically how convex a large sum of sets with the same convexity index can be made.
\end{remark}

If equality holds in Proposition \ref{lb_candidate}, then we can obtain a formula for $c^{*}(A)$ when $A$ is compact in $\mathbb{R}$.

\begin{proposition}
     Let $A\subseteq\mathbb{R}$ be a non-empty, compact set such that $c(A)<1$. Assume that the inequality in Proposition \ref{lb_candidate} is in fact an equality. Then
        \begin{equation*}
            c^{*}(A)=\textup{Int}\left[\log_{2}\left(\frac{1}{1-c(A)}\right)\right],
        \end{equation*}
     where for any $x\in\mathbb{R}$, we define $\textup{Int}[x]:=\inf\{k\in\mathbb{Z}:k\geq x\}$. Moreover, if $c(A)=1$, then $c^{*}(A)=\infty$.
\end{proposition}

\begin{proof}
     Assuming equality holds in Proposition \ref{lb_candidate}, then to prove this proposition, we need to solve the inequality
        \begin{equation*}
            1-\sum_{j=1}^{k}2^{j-1}r\left(c(A)\right)\leq0.
        \end{equation*}
        Using the formula $\sum_{j=1}^{k}2^{j-1}=2^{k}-1$, we solve the inequality to get
        \begin{equation*}
            \log_{2}\left(\frac{2}{1-c(A)}\right)\leq k.   
        \end{equation*}
        From this, we see that 
        \begin{equation*}
            N_{c(A)}=\textup{Int}\left[\log_{2}\left(\frac{2}{1-c(A)}\right)\right]=1+\textup{Int}\left[\log_{2}\left(\frac{1}{1-c(A)}\right)\right],
        \end{equation*}
    which completes the proof (If $c(A)=1$, it is only necessary to realize that no finite number of sets with exactly two points can add to an interval).
    \end{proof}
    
\section{Fractional subadditivity of the Schneider non-convexity index}\label{fractionally_subadditive}
\subsection{Preliminary information}
We will start off with an important lemma about reducing to partitions with rational weights. This idea is not new, as it is originally used for this type of problem in \cite{barthe1} for the proof that Lebesgue measure is fractionally superadditive. The idea also works for the Schneider non-convexity index.
\begin{lemma}[Barthe and Madiman \cite{barthe1}]\label{rational_partitions}
    Let $m$ be a positive integer, and suppose that for each rational fractional partition $(\mathcal{G},\beta)$ of $[m]$ (i.e $\beta(S)\in\mathbb{Q}$ for each $S\in\mathcal{G}$) we have the fractional subbadditive inequality
    \begin{equation}\label{sub_additive_rational_assumption}
        c\left(\sum_{i=1}^{m}A_i\right)\leq\sum_{S\in\mathcal{G}}\beta(S)c\left(\sum_{i\in S}A_i\right).
    \end{equation}
    Then the fractional subadditive inequality (\ref{sub_additive_rational_assumption}) holds for any fractional partition of $[m]$.
\end{lemma}
\begin{remark}
The significance of Lemma \ref{rational_partitions} is that it allows us to reduce proving the fractional subadditive property for arbitrary partitions to proving the property for the partitions which have rational weights. If the weights are all rational, then they have a common denominator $q\in\mathbb{Z}_{+}$. It follows that in the rational case, we need to prove that
\begin{equation*}
   q\cdot c(A_1+\dots+A_m)\leq\sum_{j=1}^{t}c\left(\sum_{i\in S_j}A_i\right),
\end{equation*}
where the sets $S_j$ are an enumeration of the sets $S\in\mathcal{G}$, possibly with repetition, $q\in\mathbb{Z}^{+}$ is some positive integer, and each $i\in[m]$ belongs to exactly $q$ of the sets $S_j$.
\end{remark}
The other important result is a corollary of Theorem \ref{strong bound} stated in the introduction. The following lemma is essentially given in \cite{fradelizi1}, but we will provide the proof here anyway for completion.
\begin{lemma}[Fradelizi et al. \cite{fradelizi1}]\label{strong_bound_rewrite}
    Let $m$ be a positive integer, and let $A_1,\dots,A_m$ be non-empty, compact sets in $\mathbb{R}^{n}$. Then for any $S,T\subseteq[m]$, we have
    \begin{equation*}
        c\left(\sum_{i\in S\cup T}A_i\right)\leq\max\left\{c\left(\sum_{i\in S}A_i\right),c\left(\sum_{i\in T}A_i\right)\right\}.
    \end{equation*}
\end{lemma}
\begin{proof}
    Suppose that $S\cap T=\varnothing$. Then
    \begin{equation*}
        \sum_{i\in S\cup T}A_i=\sum_{i\in S}A_i+\sum_{i\in T}A_i.
    \end{equation*}
    We then use the inequality $c(A+B)\leq\max\{c(A),c(B)\}$ with $A:=\sum_{i\in S}A_i$ and $B:=\sum_{i\in T}A_i$. Next, suppose that $S\backslash T=\varnothing$. Then $S\subseteq T$ and we have 
    \begin{equation*}
        \sum_{i\in S\cup T}A_i=\sum_{i\in T}A_i,
    \end{equation*}
    from which the inequality immediately follows. This reasoning also covers the case where $T\backslash S=\varnothing$. Finally, assume that none of the sets $S\cap T$, $S\backslash T$, or $T\backslash S$ is the empty set. Using the notation in Theorem \ref{strong bound}, let $A=\sum_{i\in S\backslash T}A_i$, $B=\sum_{i\in S\cap T}A_i$, and $C=\sum_{i\in T\backslash S}A_i$. Using the relation $S\cup T=(S\backslash T)\cup(S\cap T)\cup(T\backslash S)$, which is a disjoint union, we get 
    \begin{equation*}
    \begin{split}
        A+B+C&=\sum_{i\in S\backslash T}A_i+\sum_{i\in S\cap T}A_i+\sum_{i\in T\backslash S}A_i=\sum_{i\in S\cup T}A_i,\\
        A+B&=\sum_{i\in S\backslash T}A_i+\sum_{i\in S\cap T}A_i=\sum_{i\in S}A_i,\\
        B+C&=\sum_{i\in T\backslash S}A_i+\sum_{i\in S\cap T}A_i=\sum_{i\in T}A_i.
    \end{split}
    \end{equation*}
    Now, combining the above three sums with Theorem \ref{strong bound} completes the proof.
\end{proof}

We will also look into the equality conditions of the fractional subadditive inequalities in the case when the dimension is $n=1$. There is a certain convention introduced in the paper \cite{meyer1} which can be used when the function under consideration is translation invariant. To begin with, we will give a simple example.

\begin{example}\label{translation_example}
Consider the first of the non-trivial fractional subadditive inequalities
\begin{equation*}
    c(A_1+A_2+A_3)\leq \frac{1}{2}c(A_1+A_2)+\frac{1}{2}c(A_1+A_3)+\frac{1}{2}c(A_2+A_3).
\end{equation*}
Suppose, for example, that the set $A_3$ contains only one point. Then using the translation invariance of the Schneider non-convexity index, the inequality can be simplified to
\begin{equation*}
    c(A_1+A_2)\leq \frac{1}{2}c(A_1+A_2)+\frac{1}{2}c(A_1)+\frac{1}{2}c(A_2).
\end{equation*}
We can simplify things even further to get
\begin{equation*}
    c(A_1+A_2)\leq c(A_1)+c(A_2),
\end{equation*}
which is a fractional sub-additive inequality on $3-1=2$ sets. 
\end{example}

The idea demonstrated in Example \ref{translation_example} holds in the general case, as is shown in \cite{meyer1} for volumes, but is also true for the Schneider non-convexity index since the proof only relies on translation invariance.
\begin{lemma}[\cite{meyer1}]\label{translated_partition}
    Let $m$ be a positive integer, let $(\mathcal{G},\beta)$ be a fractional partition of $[m]$, and let $A_1,\dots,A_m$ be non-empty, compact sets in $\mathbb{R}^n$. Suppose that exactly $k< m$ of the sets $A_1,\dots,A_m$ have $\textup{card}(A_i)=1$. Let $B_1,\dots,B_{m-k}$ denote those sets $A_i$ for which $\textup{card}(A_i)\geq2$. Then there exists a fractional partition $(\mathcal{G}^{*},\beta^{*})$ of $[m-k]$ such that 
    \begin{equation}\label{translated_equation}
        \sum_{T\in\mathcal{G}^{*}}\beta^{*}(T)c\left(\sum_{i\in T}B_i\right)=\sum_{S\in\mathcal{G}}\beta(S)c\left(\sum_{i\in S}A_i\right).
    \end{equation}
\end{lemma}
\begin{remark}
The partition $(\mathcal{G}^{*},\beta^{*})$ in Lemma \ref{translated_partition} is referred to as the \textit{translated} partition. Let $A_1,\dots,A_m$ be sets in $\mathbb{R}^n$. Suppose that we want to find when (\ref{sub_additive_rational_assumption}) is equality. Using Lemma \ref{translated_partition}, we only need to consider the case where each set $A_i$ has $\textup{card}(A_i)\geq2$. For in the case that some of the sets $A_i$ contain a single point we have by translation invariance of $c$, and assuming (which we will prove) that $c$ is fractionally subadditive
\begin{equation*}
    c\left(\sum_{i=1}^{m}A_i\right)=c\left(\sum_{i=1}^{m-k}B_i\right)\leq\sum_{T\in\mathcal{G}^{*}}\beta^{*}(T)c\left(\sum_{i\in T}B_i\right)=\sum_{S\in\mathcal{G}}\beta(S)c\left(\sum_{i\in S}A_i\right).
\end{equation*}
It follows that the fractional subadditive inequality for $(\mathcal{G}^{*},\beta^{*})$ is equality if and only if the fractional subadditive inequality for $(\mathcal{G},\beta)$ is equality.
\end{remark}

\subsection{Proof of the fractional subadditivity}
\begin{theorem}
    Let $A_1,\dots,A_m$ be non-empty, compact sets in $\mathbb{R}^n$, and let $(\mathcal{G},\beta)$ be a fractional partition of $[m]$. Then
    \begin{equation*}  c\left(\sum_{i=1}^{m}A_i\right)\leq\sum_{S\in\mathcal{G}}\beta(S)c\left(\sum_{i\in S}A_i\right).
    \end{equation*}
    Moreover, in the case of dimension $n=1$, equality holds if and only if exactly one of the properties holds:
    \begin{enumerate}
        \item The translated partition is trivial (i.e $\mathcal{G^{*}}=[k]$ and $\beta^{*}([k])=1$).
        \item The translated partition is non-trivial and for each $S\in\mathcal{G}$ with $\beta(S)>0$, the set $\sum_{i\in S}A_i$ is an interval.
    \end{enumerate}
\end{theorem}
\begin{proof}
By Lemma \ref{rational_partitions} we need to prove that
\begin{equation*}
    q\cdot c(A_{[m]})\leq\sum_{j=1}^{t}c(A_{S_j}),
\end{equation*}
where $q\geq1$ is an integer, the sets $\{S_j\}_{j=1}^{t}$ are an enumeration of the sets $S\in\mathcal{G}$, and each $i\in[m]$ belongs to exactly $q$ of the sets $S_j$. Then $t\geq q$. Otherwise, if $t<q$, each $i\in[m]$ belongs to fewer than $q$ sets $S_j$. Let $r\geq0$ be an integer such that $t=q+r$. It must be true that any $r+1$ of the sets $S_j$ must have a union that is equal to $[m]$. To see this, assume for contradiction that there exist sets $S_{j_1},\dots S_{j_{r+1}}$ and $i\in[m]$ such that 
\begin{equation*}
    i\notin S_{j_{1}}\cup\dots\cup S_{j_{r+1}}.
\end{equation*}
Then $i$ can belong to at most $q+r-(r+1)=q-1$ of the sets $S_j$, which is a contradiction. Now, suppose that the sets $S_j$ are arranged so that
\begin{equation*}
    c(A_{S_1})\geq c(A_{S_2})\geq\dots\geq c(A_{S_t}).
\end{equation*}
For $j\in[q-1]$ we have, by using induction on Lemma \ref{strong_bound_rewrite}
\begin{equation*}
    \begin{split}
        c(A_{[m]})&=c\left(\sum_{i\in S_j\cup[S_q\cup\dots\cup S_t]}A_i\right)\\
        &\leq\max\{c(A_{S_j}),c(A_{S_q}),\dots,c(A_{S_t})\}\\
        &=c(A_{S_j}).
    \end{split}
\end{equation*}
Using Lemma \ref{strong_bound_rewrite} again, we have
\begin{equation*}
    \begin{split}
        c(A_{[m]})&=c\left(\sum_{i\in S_q\cup\dots\cup S_{t}}A_i\right)\\
        &\leq\max\{c(A_{S_q}),\dots,c(A_{S_t})\}\\
        &\leq c(A_{S_q})+\dots+c(A_{S_t}).
    \end{split}
\end{equation*}
Now combine the above $q$ inequalities to get
\begin{equation*}
    q\cdot c(A_{[m]})\leq\sum_{j=1}^{q-1}c(A_{S_j})+\sum_{j=q}^{q+r}c(A_{S_j})=\sum_{j=1}^{t}c(A_{S_j}).
\end{equation*}
This completes the proof that $c$ is fractionally subadditive.

Now we will prove the equality conditions for $n=1$. Suppose that the fractional partition is non-trivial. By Lemma \ref{translated_partition} and translation invariance, we may assume that $\textup{card}(A_i)\geq2$ for each $i\in[m]$. We want to show that
\begin{equation*}
    c(A_{[m]})=\sum_{S\in\mathcal{G}}\beta(S)c(A_S)
\end{equation*}
if and only if each set $A_S$ is an interval, where $A_1,\dots,A_m$ each have $\text{card}(A_i)\geq2$. Since the weights are rational by assumption the equality can be written as
\begin{equation*}
    q\cdot c(A_{[m]})=\sum_{j=1}^{q+r}c(A_{S_j}),
\end{equation*}
where $r\geq1$ is an integer (if $r=0$, then the partition is trivial) and the sets $S_j$ are arranged such that 
\begin{equation*}
    c(A_{S_1})\geq c(A_{S_2})\geq\dots\geq c(A_{S_{q+r}}).
\end{equation*}
Recall that any $r+1$ of the sets $S_j$ must have a union that is equal to $[m]$. Then since we are assuming equality holds
\begin{equation*}
    \begin{split}
    c(A_{[m]})&=c\left(\sum_{S_q\cup S_{q+1}\cup\dots\cup S_{q+r}}A\right)\\
    &=\max\{c(A_{S_q}),c(A_{S_{q+1}}),\dots,c(A_{S_{q+r}})\}\\
    &=c(A_{S_q}).
    \end{split}
\end{equation*}
But by observing the proof of the inequality we must have
\begin{equation*}
    c(A_{[m]})=c(A_{S_q})+c(A_{S_{q+1}})+\dots+c(A_{S_{q+r}}).
\end{equation*}
Then
\begin{equation*}
    c(A_{S_{q+1}})=\dots=c(A_{S_{q+r}})=0.
\end{equation*}
It follows that each of the sets $S_{j}$, $q+1\leq j\leq q+r$ are intervals. Again, observing the proof of the inequality we have
\begin{equation*}
    c(A_{[m]})=c(A_{S_1})=\dots=c(A_{S_{q}}).
\end{equation*}
If $S_{q+1}\cup\dots\cup S_{q+r}=[m]$, then $c(A_{[m]})=0$ and it follows that $c(A_{S_j})=0$ for each $j\in[q+r]$ which proves what we need. Now, assume that $c(A_{[m]})>0$ and that $S_{q+1}\cup\dots\cup S_{q+r}\neq[m]$. We claim that there exists $j\in[q]$ such that $(S_{q+1}\cup\dots\cup S_{q+r})\backslash S_j\neq\varnothing$. If not, then for each $j\in[q]$ we have $S_{q+1}\cup\dots\cup S_{q+r}\subseteq S_j$, and some $i\in S_{q+1}\cup\dots\cup S_{q+r}$ belongs to at least $q+1$ sets $S_j$, which is a contradiction. Now, choose $j\in[q]$ such that $(S_{q+1}\cup\dots\cup S_{q+r})\backslash S_j\neq\varnothing$. Recall that $c(A_{S_j})=0$ for each $j\geq q+1$. It follows that $A_{S_{q+1}}+\dots+A_{S_{q+r}}=\text{conv}(A_{S_{q+1}}+\dots+A_{S_{q+r}})$. Then
\begin{equation*}
    \begin{split}
        A_{[m]}&=A_{S_j\cup(S_{q+1}\cup\dots\cup S_{q+r})}\\
        &=A_{S_j\backslash(S_{q+1}\cup\dots\cup S_{q+r})}+A_{S_{q+1}\cup\dots\cup S_{q+r}}\\
        &=A_{S_j\backslash(S_{q+1}\cup\dots\cup S_{q+r})}+\text{conv}(A_{S_{q+1}\cup\dots\cup S_{q+r}})\\
        &=A_{S_j\backslash(S_{q+1}\cup\dots\cup S_{q+r})}+\text{conv}(A_{S_j\cap(S_{q+1}\cup\dots\cup S_{q+r})})\\
        &+\text{conv}(A_{(S_{q+1}\cup\dots\cup A_{S_{q+r}})\backslash S_j}).
    \end{split}
\end{equation*}
Now, the set $A_{S_j}$ cannot be an interval since $c(A_{S_j})=c(A_{[m]})>0$. Then since $\text{conv}(A_{(S_{q+1}\cup\dots\cup A_{S_{q+r}})\backslash S_j})$ is an interval with positive length it follows that 
\begin{equation*}
    \begin{split}
        c(A_{[m]})&\leq c(A_{S_j}+\text{conv}(A_{(S_{q+1}\cup\dots\cup A_{S_{q+r}})\backslash S_j}))\\
        &<c(A_{S_j}).
    \end{split}
\end{equation*}
This is a contradiction. Therefore the assumption that $c(A_{[m]})>0$ is false. It follows that $c(A_{[m]})=0$ and that $c(A_{S_j})=0$ for each $j\in[q]$. Then each of the sets $A_{S_j}$ is an interval and the equality conditions are complete.
\end{proof}

\begin{remark}
    We have not been able to figure out the equality conditions when the dimension is $n\geq2$. The reason for this is that in general it is not true that if $A$ is compact and $B$ is convex (containing more than one point), then $c(A+B)<c(A)$. For example, if $A=([0,1]\cup[2,3])\times\{0\}$ and $B=\{0\}\times[0,1]$, then $c(A)=\frac{1}{3}$ and it can be verified that also $c(A+B)=\frac{1}{3}$.
\end{remark}

\section{Regions defined by set functions}
\subsection{The Schneider region}\label{schneider_region}
The Schneider region is defined in a similar way as the Lyusternik region.
\begin{definition}
    For positive integers $n,m$ recall that the set $\mathcal{K}_n(m)$ represents the collection of all $m$-tuples $A=(A_1,\dots,A_m)$ of non-empty, compact sets in $\mathbb{R}^n$. Analogous to the Lyusternik region, for a set $A\in \mathcal{K}_n(m)$ we define the function $c_A:2^{[m]}\rightarrow[0,\infty)$ by
    \begin{equation*}
        c_A(S):=c\left(\sum_{i\in S}A_i\right).
    \end{equation*}
    Then the $(n,m)$-Schneider region is defined by
    \begin{equation*}
        S_n(m):=\{c_A:A\in \mathcal{K}_n(m)\}.
    \end{equation*}
    Just as in the definition of the Lyusternik region, we can associate an element $c_A\in S_{n}(m)$ with a vector in $\mathbb{R}^{2^m}$ by first fixing an ordering $S_1,\dots,S_{2^m}$ of the sets in $2^{[m]}$, and then associating $c_A$ with the vector $(c_A(S_1),\dots,c_A(S_{2^m}))$.
\end{definition}

\begin{remark}\label{special_notation}
To make the characterization of $S_1(2)$ easier to state, we introduce some notation. For a set $S\subseteq[m]$ define
\begin{equation*}
    S_n(m,S):=\{c_A\in S_n(m): \textup{card}(A_i)\geq2 \textup{ iff } i\in S\}.
\end{equation*}
The set $S_n(m,S)$ represents the piece of $S_n(m)$ for which exactly the sets $A_i$ with $i\in S$ contain two or more points, and any other sets $A_i$ each consist of a single element. We will also use the following conventions for the upper and lower bounds on $c(A_1+A_2)$. Let $c_1,c_2\in[0,1]$. Define
\begin{equation*}
    \begin{split}
        L&:=L(c_1,c_2)=\max\left(0,\min_{\sigma\in S_2}\frac{1-r(c_{\sigma(1)})-2r(c_{\sigma(2)})}{3+r(c_{\sigma(1)})+2r(c_{\sigma(2)})}\right),\\
        M&:=M(c_1,c_2)=\max(c_1,c_2).
    \end{split}
\end{equation*}
\end{remark}

\subsubsection{The largest gap in a sumset}
We will need the fact that $G(A_1+A_2)\leq \max\{G(A_1),G(A_2)\}$, which follows from the definition. In the lemma below, we prove an even better upper bound for the bound on the largest gap in the sum of $k$ sets, which is much more than we need, but is interesting.
\begin{lemma}\label{largest_gap_large_sum_exact}
    Let $k\geq1$ be an integer, and let $A_1,\dots,A_k$ be non-empty, compact sets in $\mathbb{R}$ such that 
    \begin{equation*}
        G(A_1)\geq G(A_2)\geq\dots\geq G(A_k).
    \end{equation*}
    Then
    \begin{equation*}
        G\left(\sum_{j=1}^{k}A_j\right)\leq\max_{1\leq r\leq k}\left(G(A_{r})-\sum_{j=r+1}^{k}G(A_j)\right).
    \end{equation*}
    Moreover, this is the best upper bound in the sense that if $g_1\geq g_2\geq\dots\geq g_k\geq0$, then
    \begin{equation*}
        \max_{\substack{A\in\mathcal{K}_{1}(k)\\G(A_j)=g_j}}G\left(\sum_{j=1}^{k}A_j\right)=\max_{1\leq r\leq k}\left(g_r-\sum_{j=r+1}^{k}g_j\right).
    \end{equation*}
\end{lemma}
\begin{proof}
    We will prove this result by induction on the number of sets $k$. By translation invariance, we may assume that $\min(A_i)=0$ for each $i\in[k]$. For the initial step, where $k=1$, we just observe that $G(A_1)=\max\{G(A_1)\}$. For the induction step, assume that the upper bound is true for sums of $k-1$ sets, and let $A_1,\dots,A_k$ be compact sets such that
    \begin{equation*}
        G(A_1)\geq\dots\geq G(A_k).
    \end{equation*}
    By induction we have
    \begin{equation*}
        G\left(\sum_{i=2}^{k}A_i\right)\leq\max_{2\leq r\leq k}\left(G(A_{r})-\sum_{j=r+1}^{k}G(A_j)\right).
    \end{equation*}
    We may assume that $\textup{card}(A_1)<\infty$. Otherwise, remove all but finitely many points from $A_1$ to achieve a new set $A_1^{\prime}$ such that $\textup{card}(A_1^{\prime})<\infty$ and $G(A_1)=G(A_1^{\prime})$. Then observe that $G(A_1+A_2)\leq G(A_1^{\prime}+A_2)$, so that if we prove the upper bound with $A_1^{\prime}$, we also have verified the upper bound with $A_1$. With this finite assumption on $A_1$, we may write $A_1=\{x_0,x_1,\dots,x_m\}$, where $x_0:=0<x_1<\dots<x_m:=a_1$. Then denoting $A:=A_2+\dots+A_k$, we have
    \begin{equation}\label{large_gap_sum}
        A_1+(A_2+\dots+A_k)=\bigcup_{j=0}^{m}(A+\{x_j\}).
    \end{equation}
    For each $i\in[k]$ define $a_i:=\max(A_i)$. It is enough, by (\ref{large_gap_sum}), to check the largest gap in each of the intervals $[x_{j-1},x_j+\sum_{i=2}^{k}a_i]$, for $j\in[m]$. Fix $j\in[m]$. If $x_j>x_{j-1}+\sum_{i=2}^{k}a_i$, then using the fact that $G(A_i)\leq a_i$, we calculate
    \begin{equation*}
    \begin{split}
        x_j-(x_{j-1}+\sum_{i=1}^{k}a_i)&=(x_j-x_{j-1})-\sum_{i=2}^{k}a_i\\
        &\leq G(A_1)-\sum_{i=2}^{k}a_i\\
        &\leq G(A_1)-\sum_{i=2}^{k}G(A_i).
    \end{split}
    \end{equation*}
    It follows that
    \begin{equation*}
       G\left((A_1+A)\cap[x_{j-1},x_j+\sum_{i=2}^{k}a_i]\right)\leq\max\left\{G(A),G(A_1)-\sum_{i=2}^{k}G(A_i)\right\}.
    \end{equation*}
    If $x_j\leq x_{j-1}+\sum_{i=2}^{k}a_i$, then 
    \begin{equation*}
         G\left((A_1+A)\cap[x_{j-1},x_j+\sum_{i=2}^{k}a_i]\right)\leq G(A).
    \end{equation*}
    This is true for any $j\in[m]$. Therefore
    \begin{equation*}
    \begin{split}
        G\left(\sum_{i=1}^{k}A_i\right)&=\max_{j\in[m]}G\left((A_1+A)\cap[x_{j-1},x_j+\sum_{i=2}^{k}a_i]\right)\\
        &\leq\max\left\{G(A),G\left(A_1\right)-\sum_{i=2}^{k}G(A_i)\right\}\\
        &\leq\max_{1\leq r\leq k}\left(G(A_{r})-\sum_{j=r+1}^{k}G(A_j)\right).
    \end{split}
    \end{equation*}
    This proves the upper bound. It remains to show that this is the best upper bound. To see this, let $g_1\geq g_2\geq\dots\geq g_k\geq0$ and define $A_j=\{0,g_j\}$. Suppose we first calculate $G(A_{k-1}+A_k)$. We have
    \begin{equation}\label{big_gap}
        A_{k-1}+A_k=\{0,g_{k},g_{k-1},g_{k-1}+g_{k}\}.
    \end{equation}
    By observation we see that
    \begin{equation*}
        G(A_{k-1}+A_k)=G((A_{k-1}+A_k)\cap[0,g_{k-1}])=\max\{g_k,g_{k-1}-g_k\}.
    \end{equation*}
    Notice that the largest gap comes from the first three terms in (\ref{big_gap}), and if we add terms between the third and fourth terms, the largest gap remains unchanged. Now, consider the set $A_{k-2}+A_{k-1}+A_k$. We can write
    \begin{equation*}
        A_{k-2}+A_{k-1}+A_k=(A_{k-1}+A_{k})\cup(\{g_{k-2}\}+(A_{k-1}+A_{k}).
    \end{equation*}
    If $g_{k-2}\leq g_{k-1}+g_k$, then we still take the maximum over the set $\{g_k,g_{k-1}-g_k\}$, since by the above observation, the new set will consist of two copies of $A_{k-1}+A_k$, which overlap somewhere in $[g_{k-1},g_{k-1}+g_k]$, therefore not disturbing the largest gap from the first sum, but also not adding any larger gaps. If $g_{k-2}>g_{k-1}+g_{k}$, then we have two disjoint copies of $A_{k-1}+A_k$ separated by a distance of $g_{k-2}-g_{k-1}-g_{k}$. It follows from this that
    \begin{equation*}
        G(A_{k-2}+A_{k-1}+A_{k})=\max\{g_k,g_{k-1}-g_k,g_{k-2}-g_{k-1}-g_{k}\}.
    \end{equation*}
    Also, we observe that we only need to use the points less than or equal to $g_{k-2}$ to achieve this. Continuing this process inductively, we see that after $t$ iterations the largest gap is
    \begin{equation*}
        G\left(\sum_{j=k-t}^{k}A_j\right)=\max_{k-t\leq r\leq k}\left(g_r-\sum_{j=r+1}^{k}g_j\right).
    \end{equation*}
    By setting $t=k-1$, we get the result.
\end{proof}

\begin{remark}
    Lemma \ref{largest_gap_large_sum_exact} can also be interpreted as the best upper bound for the Hausdorff distance to convex hull. For a compact set $A\subseteq \mathbb{R}^n$, define the Hausdorff distance between $A$ and $\textup{conv}(A)$ by 
    \begin{equation*}
        d(A):=\inf\{\lambda>0:\textup{conv}(A)\subseteq A+rB_{2}^n\},
    \end{equation*}
    where $B_2^n$ is the unit ball in $\mathbb{R}^n$ with respect to the $l_2$ norm. In dimension $n=1$, we will verify that $d(A)=\frac{1}{2}G(A)$. To see this, observe that since $A+\lambda[-1,1]+\lambda=A+\lambda[0,2]$, we have
    \begin{equation*}
        \begin{split}
            d(A)&=\inf\{\lambda>0:\textup{conv}(A)\subseteq A+\lambda[-1,1]\}\\
            &=\inf\{\lambda>0:A+\lambda[-1,1]\textup{ is convex }\}\\
            &=\inf\{\lambda>0:A+\lambda[0,2]\textup{ is convex }\}\\
            &=\frac{1}{2}\inf\{\lambda>0:A+\lambda[0,1]\textup{ is convex }\}=\frac{1}{2}G(A).
        \end{split}
    \end{equation*}
    This can be seen geometrically too. If $A$ is compact, then it is well known that 
    \begin{equation*}
        d(A)=\sup_{x\in\textup{conv}(A)}d(x,A),
    \end{equation*}
    where $d(x,A)$ is the smallest distance from $x$ to the set $A$. Intuitively, in dimension $n=1$ this supremum will be attained when the point $x$ lies exactly half way through the largest gap inside the set $A$, from which we get the above claim. Therefore another way to state Lemma \ref{largest_gap_large_sum_exact} is that if $d(A_1)\geq\dots\geq d(A_k)$, then
    \begin{equation*}
        d\left(\sum_{j=1}^{k}A_j\right)\leq\max_{1\leq r\leq k}\left(d(A_{r})-\sum_{j=r+1}^{k}d(A_j)\right).
    \end{equation*}
    This is the best upper bound in the same sense as given above.
\end{remark}

\subsubsection{Characterization for two sets in one-dimension}
We now give the characterization of $S_1(2)$. Refer to Remark \ref{special_notation} for notation.
\begin{theorem}\label{char_thm}
    In the case $n=1$, $m=2$ we have that 
    \begin{equation*}
        S_1(2)=S_1(2,\{1,2\})\cup S_1(2,\{1\})\cup S_1(2,\{2\})\cup S_1(2,\varnothing),
    \end{equation*}
    where
    \begin{enumerate}
        \item $S_1(2,\{1,2\})=\{(0,c_1,c_2,c_{12})\in\mathbb{R}^{4}_{+}: (c_1,c_2)\in[0,1]^{2}, c_{12}\in[L,M)\}$.
        \item $S_1(2,\{1\})=\{(0,c_1,0,c_{12})\in\mathbb{R}^{4}_{+}:c_1\in[0,1],c_{12}=c_1\}$.
        \item $S_1(2,\{2\})=\{(0,0,c_2,c_{12})\in\mathbb{R}^{4}_{+}:c_2\in[0,1],c_{12}=c_2\}$.
        \item $S_1(2,\varnothing)=\{(0,0,0,0)\}$.
    \end{enumerate}
\end{theorem}
\begin{proof}
    The characterization of $S_{1}(2,\varnothing)$ follows from the fact that the only sets being considered are point sets. The characterization of $S_{1}(2,\{1\})$ follows from the fact that if $A_1$, $A_2$ are sets in $\mathbb{R}$ such that $\textup{card}(A_2)=1$, then $c(A_1+A_2)=c(A_1)$. By the same reasoning we can verify the characterization for $S_{1}(2,\{2\})$. It remains to verify the characterization for $S_1(2,\{1,2\})$. To prove this, let $c_A\in S_1(2,\{1,2\})$. Then $A=(A_1,A_2)$ is a pair of nonempty, compact sets which contain two or more points each. Define $c_1:=c(A_1)$ and $c_2:=c(A_2)$. By Theorem \ref{Lower_bound}, $c(A_1+A_2)\geq L$. To verify the upper bound we use Lemma \ref{largest_gap_large_sum_exact} and the fact that $|\textup{conv}(A_i)|>0$ to get
    \begin{equation*}
        c(A_1+A_2)\leq\frac{\max\left\{G(A_1),G(A_2)\right\}}{|\textup{conv}(A_1)|+|\textup{conv}(A_2)|}<\max\left\{\frac{G(A_1)}{|\textup{conv}(A_1)|},\frac{G(A_2)}{|\textup{conv}(A_2)|}\right\}=:M.
    \end{equation*}
    Hence, $c(A_1+A_2)\in [L,M)$. Now, let $(0,c_1,c_2,c_{12})\in \mathbb{R}^{4}_{+}$ such that $(c_1,c_2)\in[0,1]^{2}$, and $c_{12}\in [L,M)$. We will show that there exist compact sets $A_1, A_2$, which are not point sets, and $c(A_1)=c_1$, $c(A_2)=c_2$, $c(A_1+A_2)=c_{12}$. For $j\in[2]$ and $\sigma\in S_2$, set $A_{\sigma(j)}=[0,r(c_{\sigma(j)})]\cup[1,1+r(c_{\sigma(j)})]$. Then $c(A_{\sigma(j)})=c_{\sigma(j)}$. First, choose $\sigma_{*}\in S_2$ such that 
    \begin{equation*}
        \frac{1-r(c_{\sigma_{*}(1)})-2r(c_{\sigma_{*}(2)})}{3+r(c_{\sigma_{*}(1)})+2r(c_{\sigma_{*}(2)})}=\min_{\sigma\in S_2}\frac{1-r(c_{\sigma(1)})-2r(c_{\sigma(2)})}{3+r(c_{\sigma(1)})+2r(c_{\sigma(2)})}.
    \end{equation*}
    If $m\in[2,\infty)$, then the sumset $A_{\sigma_{*}(1)}+mA_{\sigma_{*}(2)}$ has as its largest gap
    \begin{equation*}
        G=m-1-r(c_{\sigma_{*}(1)})-mr(c_{\sigma_{*}(2)}).
    \end{equation*}
    Then for $m\in[2,\infty)$, the convexity index of the above sumset is
    \begin{equation*}
        c(A_{\sigma_{*}(1)}+mA_{\sigma_{*}(2)})=\frac{m-1-r(c_{\sigma_{*}(1)})-mr(c_{\sigma_{*}(2)})}{1+m+r(c_{\sigma_{*}(1)})+mr(c_{\sigma_{*}(2)})}.
    \end{equation*}
    The above rational function is non-decreasing and continuous as a function of $m$. Also
    \begin{equation*}
        \lim_{m\rightarrow\infty}\frac{m-1-r(c_{\sigma_{*}(1)})-mr(c_{\sigma_{*}(2)})}{1+m+r(c_{\sigma_{*}(1)})+mr(c_{\sigma_{*}(2)})}=\frac{1-r(c_{\sigma_{*}(2)})}{1+r(c_{\sigma_{*}(2)})}=c_{\sigma_{*}(2)}.
    \end{equation*}
    Setting $f_{*}(m):=c(A_{\sigma_{*}(1)}+mA_{\sigma_{*}(2)})$, it follows that 
    \begin{equation*}
        f_{*}\left([2,\infty)\right)=\left[\frac{1-r(c_{\sigma_{*}(1)})-2r(c_{\sigma_{*}(2)})}{3+r(c_{\sigma_{*}(1)})+2r(c_{\sigma_{*}(2)})},c_{\sigma_{*}(2)}\right)=[L,c_{\sigma_{*}(2)}).
    \end{equation*}
    It may not be true that $c_{\sigma_{*}(2)}=\max\{c_1,c_2\}$. If that is the case, then choose $\sigma^{*}\in S_2$ such that $c_{\sigma^{*}(2)}=\max\{c_1,c_2\}$. Repeating the same computations that were done for $\sigma_{*}$, we find, setting $f^{*}(m):=c(A_{\sigma^{*}(1)}+mA_{\sigma^{*}(2)})$ that
    \begin{equation*}
        f^{*}\left([2,\infty)\right)=\left[\frac{1-r(c_{\sigma^{*}(1)})-2r(c_{\sigma^{*}(2)})}{3+r(c_{\sigma^{*}(1)})+2r(c_{\sigma^{*}(2)})},c_{\sigma^{*}(2)}\right)=\left[\frac{1-r(c_{\sigma^{*}(1)})-2r(c_{\sigma^{*}(2)})}{3+r(c_{\sigma^{*}(1)})+2r(c_{\sigma^{*}(2)})},M\right).
    \end{equation*}
    Then $f_{*}([2,\infty))\cup f^{*}([2,\infty))=[L,M)$. This completes the proof.
\end{proof}
\begin{remark}
    We note here that, while in general finding a complete characterization of the region $S_n(m)$ for any $n,m$ looks extremely complicated, it would be interesting to achieve a characterization of the region $S_n(2)$, for $n\geq2$. The main difficulty that has prevented us from achieving this is that we do not know the generalization of Theorem \ref{Lower_bound} for arbitrary compact sets in $\mathbb{R}^n$.
\end{remark}

\subsubsection{Two sets in high dimension}
As noted in the remark in the previous section, we do not know the characterization of $S_n(2)$ when $n\geq2$. But, we can give the characterization for a piece of this region.
\begin{proposition}\label{partial_char}
    Let $n\geq2$. Then
    \begin{equation*}
        \{(0,c_1,c_2,c_{12})\in\mathbb{R}^{4}_+:(c_1,c_2)\in[0,1]^2,c_{12}\in[0,\max(c_1,c_2)]\}\subseteq S_n(2).
    \end{equation*}
\end{proposition}

\begin{proof}
    It is enough to prove this for $n=2$. For the case where $n>2$, use the fact that $S_n(2)\subseteq S_{n+1}(2)$. This can be verified by using the fact that for a compact set $A\subseteq\mathbb{R}^n$, $c(A\times\{0\})=c(A)$. Now, for $0\leq r\leq R$, define
    \begin{equation*}
        A(r,R):=\{x\in\mathbb{R}^2:r\leq |x|\leq R\}
    \end{equation*}
    to be the annulus with smaller radius $r$ and larger radius $R$. Assume without loss of generality that $c_1=\max(c_1,c_2)$, and that $c_1>0$ (If $c_1=0$, then we could define $A_1$ and $A_2$ to be each a single point). Define $A_1:=A(c_1,1)$ and $A_2:=A(c_2,1)$. Then $c(A_1)=c_1$ and for any $m>0$, $c(mA_2)=c_2$. In particular, let $m\in(0,c_1]$. Observe that $A_1+mA_2=A(c_1-m,1+m)$ (this is a similar computation to adding two one dimensional sets together). Then $c(A_1+mA_2)=\frac{c_1-m}{1+m}$. Since this is a continuous, decreasing function of $m$, we see that as $m$ varies in the interval $(0,c_1]$, $c(A_1+mA_2)$ can take any value in the interval $[0,c_1)$. All that remains is to show that the maximum can be attained. Let $A_1$ and $A_2$ be any sets in $\mathbb{R}$ with convexity indices $c_1$ and $c_2$ respectively. Set $A_1^{\prime}=A_1\times\{0\}$ and $A_2^{\prime}=\{0\}\times A_2$. Then $c(A_j^{\prime})=c_j$ for each $j\in[2]$. The sumset is $A_1^{\prime}+A_2^{\prime}=A_1\times A_2$. Then $c(A_1^{\prime}+A_2^{\prime})=\max\{c(A_1),c(A_2)\}$, which shows the maximum convexity index is attained. This proves the proposition.
\end{proof}

\begin{remark}
    We can use Proposition \ref{partial_char} to get a partial answer to Problem \ref{high_dim_bound}. In fact, by the above proposition if $c_1,c_2\in[0,1]$ and $n\geq2$, then $F_n(c_1,c_2)=0$. 
\end{remark}

\subsection{Closure properties of the Lyusternik region}\label{closure_properties}
\subsubsection{The construction of a fractal set}
We first need to construct a fractal set.
\begin{definition}
    Let $I$ be some interval of length $|I|>0$, let $N\geq3$ be an integer, and choose $k\in\{0,\dots,N-1\}$. Denote $m:=\min(I)$, and $I_j:=[m+j|I|N^{-1},m+(j+1)|I|N^{-1}]$. Then
    \begin{equation*}
        I=I_0\cup I_1\cup\dots\cup I_{N-1}.
    \end{equation*}
    Define $F_k(I):=I_{0}\cup\dots\cup I_{k}$. That is, we divide the interval $I$ into exactly $N$ sub-intervals of equal length, and keep only the first $k+1$ sub-intervals. We will define the fractal set $C(N,k)$ inductively as follows.
    \begin{enumerate}
        \item Let $C_1:=F_k([0,1])$. Then we can write
        \begin{equation*}
            C_1=I_{0}\cup I_{1}\cup\dots\cup I_{k},
        \end{equation*}
        where each interval $I_{t}$ has length $\frac{1}{N}$.
        \item For $j=2$ define $C_2:=F_k(I_0)\cup\dots\cup F_k(I_k)$. That is,
        \begin{equation*}
            C_2=\bigcup_{\alpha \in S_2}I_{2,\alpha},
        \end{equation*}
        where each interval $I_{2,\alpha}$ has length $\frac{1}{N^{2}}$.
        \item For $j-1\geq2$, suppose that we have constructed $C_{j-1}$ so that 
        \begin{equation*}
            C_{j-1}=\bigcup_{\alpha\in S_{j-1}}I_{j-1,\alpha},
        \end{equation*}
        where the set $S_{j-1}$ is some indexing set, and the intervals $I_{j-1,\alpha}$ each have length $\frac{1}{N^{j-1}}$. Define
        \begin{equation*}
            C_j:=\bigcup_{\alpha\in S_{j-1}}F_k(I_{j-1,\alpha})=\bigcup_{\alpha\in S_j}I_{j,\alpha}
        \end{equation*}
        so that the intervals $I_{j,\alpha}$ each have length $\frac{1}{N^{j}}$. 
    \end{enumerate}
    With this inductive construction, we can define
    \begin{equation*}
        C(N,k):=\bigcap_{j=1}^{\infty}C_j.
    \end{equation*}
\end{definition}
We now show that the set $C(N,k)$ can be written as the subset of $[0,1]$ which consists of exactly those numbers possessing a base $N$ representation that only allows the digits $0,\dots,k$.
\begin{lemma}
    The set $C(N,k)$ is a compact set and has the representation
    \begin{equation}\label{nice_representation}
        C(N,k)=\left\{\sum_{j=1}^{\infty}\frac{x_j}{N^j}:x_j\in\{0,1,\dots,k\}\right\}.
    \end{equation}
\end{lemma}
\begin{proof}
    Since $C(N,k)$ is the intersection of closed sets, it is a closed set. Therefore $C(N,k)$ is a closed and bounded subset of $\mathbb{R}$, so it is compact. Denote the right side of (\ref{nice_representation}) by $S$. We want to prove that $C(N,k)=S$. Let $x\in C(N,k)$. Then $x\in C_j$ for each $j\in\mathbb{N}$. In particular, $x\in C_1=\bigcup_{j=0}^{k}I_j$. Then there exists $x_1\in\{0,\dots,k\}$ such that $x\in J_1:=I_{x_1}$. Denote the left endpoint of $J_1$ by $L_1$. Since $x\in I_{x_1}$, we have $|L_1-x|\leq\frac{1}{N}$. Now, suppose by induction that we have found integers $x_1,\dots,x_r\in\{0,\dots,k\}$ and intervals $J_{j}:=I_{j,x_j}\subset C_j$ for $j\in[r]$ with length $|J_j|=\frac{1}{N^{j}}$ such that $x\in\bigcap_{j=1}^{r}J_j$, and $L_j:=\frac{x_j}{N^{j}}+L_{j-1}$ are the left endpoints of the intervals $J_j$ which satisfy $|L_j-x|\leq\frac{1}{N^{j}}$. Since $J_r\subset C_r$, we have $x\in F_k(I_{x_r})=\bigcup_{\alpha=0}^{k}I_{r,\alpha}$. Then there exists $x_{r+1}\in\{1,\dots,k\}$ such that $x\in J_{r+1}:=I_{r+1,x_{r+1}}$. The left endpoint of the interval $J_{r+1}$ is given by $L_{r+1}=L_r+\frac{x_{r+1}}{N^{r+1}}$, and satisfies $|L_{r+1}-x|\leq\frac{1}{N^{r+1}}$. In this way, we inductively construct a sequence $\{x_j\}_{j=1}^{\infty}$ of integers in $\{0,\dots,k\}$ such that for each $r\in\mathbb{N}$, $|\sum_{j=1}^{r}\frac{x_j}{N^{j}}-x|\leq\frac{1}{N^{r}}$. It follows that $x=\sum_{j=1}^{\infty}\frac{x_j}{N^{j}}$, which implies that $x\in S$. For the other direction let $x\in S$. Then $x=\lim_{r\rightarrow\infty}\sum_{j=1}^{r}\frac{x_j}{N^{j}}$. For each $r$, the partial sum $\sum_{j=1}^{r}\frac{x_j}{N^{j}}$ is an element of $C(N,k)$. This follows from the fact that each partial sum is a left endpoint of an interval $I_{r,\alpha}$ for sum $\alpha\in S_r$, and therefore belongs to $C(N,k)$, since $C(N,k)$ contains all left endpoints of the intervals $I_{j,\alpha}$ for each $j\in\mathbb{N}$ and $\alpha\in S_j$. Since $C(N,k)$ is a compact set, this implies that $x\in C(N,k)$. Therefore $C(N,k)=S$.
\end{proof}
Next we observe some nice properties.
\begin{lemma}\label{properties_fractal_set}
The set $C(N,k)$ has the following properties.
\begin{enumerate}
    \item For any $0\leq k\leq N-2$, we have $|C(N,k)|=0$.
    \item If $k=N-1$, then $C(N,N-1)=[0,1]$.
    \item If $k+l\leq N-1$, then $C(N,k)+C(N,l)=C(N,k+l)$.
    \item The diameter of $C(N,k)$ is $\textup{diam}\left(C(N,k)\right)=\frac{k}{N-1}$.
\end{enumerate}
\end{lemma}
\begin{proof}
    For the first property, we use the definition of $C(N,k)$ and observe that in the $j$th iteration, we remove a total volume of $(N-1-k)N^{-j}$ from $(k+1)^{j-1}$ of the sub-intervals. Then the total volume removed is 
    \begin{equation*}
        |[0,1]\backslash C(N,k)|=(N-1-k)\sum_{j=1}^{\infty}\frac{(k+1)^{j-1}}{N^j}=\frac{N-k-1}{k+1}\cdot\frac{k+1}{N-k-1}=1.
    \end{equation*}
    Therefore $|C(N,k)|=0$. The second property follows from the fact that $C(N,N-1)$ is exactly the base $N$ representation of the interval $[0,1]$. To verify the third property, we observe that
    \begin{equation*}
        \begin{split}
          C(N,k)+C(N,l)&=\left\{\sum_{j=1}^{\infty}\frac{x_j}{N^j}:x_j\in\{0,1,\dots,k\}\right\}+\left\{\sum_{j=1}^{\infty}\frac{x_j}{N^j}:x_j\in\{0,1,\dots,l\}\right\}\\
          &=\left\{\sum_{j=1}^{\infty}\frac{x_j}{N^j}:x_j\in\{0,1,\dots,k+l\}\right\}=:C(N,k+l).
        \end{split}
    \end{equation*}
    To verify the fourth property we just need to compute the element of $C(N,k)$ where $x_j=k$ for each $j$. This is
    \begin{equation*}
        \sum_{j=1}^{\infty}\frac{k}{N^j}=\frac{k}{N-1}.
    \end{equation*}
    This completes the proof.
\end{proof}

\subsubsection{A note on the Brunn-Minkowski inequality}
For the example we construct for the Lyusternik region, we will need to know the equality conditions for the Brunn Minkowski inequality.
\begin{theorem}[The Brunn-Minkowski-Lyusternik inequality]
Let $A$ and $B$ be non-empty, compact sets in $\mathbb{R}^n$. Then 
\begin{equation*}
    |A+B|^{\frac{1}{n}}\geq|A|^{\frac{1}{n}}+|B|^{\frac{1}{n}},
\end{equation*}
where equality holds if and only if exactly one of the following properties hold:
\begin{enumerate}
    \item $|A+B|=0$.
    \item Exactly one of the sets $A$ and $B$ has positive measure; the other is a point.
    \item The sets $A$ and $B$ are homothetic convex bodies (in the case $n=1$, this means that $A$ and $B$ are both intervals of positive measure).
\end{enumerate}
\end{theorem}
It follows that if $A_1,A_2$ are compact in $\mathbb{R}^n$ such that $|A_1+A_2|>0$ and $|A_2|=0$, then $|A_1+A_2|^{\frac{1}{n}}=|A_1|^{\frac{1}{n}}+|A_2|^{\frac{1}{n}}$ if and only if $A_2$ is a point and $A_1$ is a compact set with positive measure. 

Aside from the reference given in the introduction, we refer to \cite{gardner1} for a detailed overview of the Brunn-Minkowski inequality, and \cite{burago1} where a proof of the equality conditions can be found for the case where the sets are compact, but not necessarily convex.

\subsubsection{The Lyusternik region is not closed}
We will show how the fractal set we just defined can be used to characterize a small piece of the Lyusternik region, and show that one of the defining volume inequalities must be a strict inequality, showing that $\Lambda_{n}(m)$ is not closed when $m\geq3$. For clarity, note that we will be associating a function $v_A\in\Lambda_n(3)$ with the vector
\begin{equation*}
    (0,|A_1|,|A_2|,|A_3|,|A_1+A_2|,|A_1+A_3|,|A_2+A_3|,|A_1+A_2+A_3|);
\end{equation*}
i.e. we choose the most natural ordering of the volumes of sumsets.
\begin{theorem}\label{lyusternik_region_clusure_property}
    Let $0<\alpha_{13}\leq\alpha_{23}<\alpha_{123}$, and let $n\geq1$ be an integer. Then
    \begin{equation*}
        (0,0,0,0,0,\alpha_{13},\alpha_{23},\alpha_{123})\in \Lambda_{n}(3).
    \end{equation*}
    Moreover, for any $\alpha>0$,
    \begin{equation*}
        (0,0,0,0,0,\alpha,\alpha,\alpha)\notin \Lambda_{n}(3),
    \end{equation*}
    which implies that $\Lambda_{n}(m)$ is not a closed set for any $m\geq3$.
\end{theorem}
\begin{proof}
    We begin by proving this for $n=1$. Let $1\leq a<b$. We will prove that there exist compact sets $B_1,B_2,B_3$ such that $|B_j|=0$ for each $j\in[3]$, $|B_1+B_2|=0$, $|B_1+B_3|=1$, $|B_2+B_3|=a$, and $|B_1+B_2+B_3|=b$. Choose integers $k_1,k_2$ and $N$ such that $k_1+k_2=N-1$, $2k_1<N-1$, and $\frac{k_1}{N-1}+a<b$. Since $1\leq a$ there exists $q_a\in\mathbb{N}$ and $r_a\in[0,1)$ such that $a=q_a+r_a$. Since $b>a$ there exists $q_a\in\mathbb{N}$ and $r_b\in[0,1)$ such that $b=q_ba+r_ba$. Define 
    \begin{equation*}
        \begin{split}
            B_1&=C(N,k_1)\cup\bigcup_{j=0}^{q_b-1}\{ja\}\cup\{q_ba+r_ba-a\},\\
            B_2&=\bigcup_{j=0}^{q_a-1}\left(\{j\}+C(N,k_1)\right)\cup(\{q_a+r_a-1\}+C(N,k_1)),\\
            B_3&=C(N,k_2).
        \end{split}
    \end{equation*}
    We immediately see that $|B_j|=0$ for each $j\in[3]$, Since $2k_1<N-1$, we have $|B_1+B_2|=0$. Also, since $k_1+k_2=N-1$ and $C(N,N-1)=[0,1]$ we have
    \begin{equation*}
        \begin{split}
            |B_1+B_3|&=|C(N,k_1+k_2)|=1,\\
            |B_2+B_3|&=\left|\bigcup_{j=0}^{q_a-1}[j,j+1]\cup[q_a+r_a-1,q_a+r_a]\right|=|[0,a]|=a.
        \end{split}
    \end{equation*}
    For the sum of all three sets first note that by the above computation, $B_2+B_3=[0,a]$. Then
    \begin{equation*}
        \begin{split}
            |B_1+B_2+B_3|&=|B_1+[0,a]|\\
            &=\left|[0,a+\textup{diam}(C(N,k_1))]\cup\bigcup_{j=0}^{q_b-1}[ja,ja+a]\cup[q_ba+r_ba-a,b]\right|\\
            &=\left|\left[0,a+\frac{k_1}{N-1}\right]\cup[0,b]\right|=|[0,b]|=b.
        \end{split}
    \end{equation*}
    Now, let $0<\alpha_{13}\leq\alpha_{23}<\alpha_{123}$. By what was just shown, there exist compact sets $B_1,B_2$ and $B_3$ such that $|B_j|=0$ for each $j\in[3]$, $|B_1+B_2|=0$, and $|B_1+B_3|=1$, $|B_2+B_3|=\alpha_{13}^{-1}\alpha_{23}$, $|B_1+B_2+B_3|=\alpha_{13}^{-1}\alpha_{123}$. Now define $A_j:=\alpha_{13}B_j$. With this choice, we have for any $S\subseteq[3]$ that
    \begin{equation*}
        \left|\sum_{i\in S}A_i\right|=\alpha_{13}\left|\sum_{i\in S}B_i\right|.
    \end{equation*}
    It follows that 
    \begin{equation*}
       (0,0,0,0,0,\alpha_{13},\alpha_{23},\alpha_{123})\in \Lambda_{1}(3). 
    \end{equation*}
     Now, suppose that $n\geq2$. Assuming that $1\leq a<b$, it follows that $1\leq a^{\frac{1}{n}}<b^{\frac{1}{n}}$. Using the one-dimensional case described above, there exist compact sets $B_i^{(n)}$ in $\mathbb{R}$ such that $|B_i^{(n)}|=0$, $|B_1^{(n)}+B_2^{(n)}|=0$, $|B_{1}^{(n)}+B_3^{(n)}|=1$, $|B_2^{(n)}+B_3^{(n)}|=a^{\frac{1}{n}}$, and $|B_1^{(n)}+B_2^{(n)}+B_3^{(n)}|=b^{\frac{1}{n}}$. Now, define the sets $B_i$ in $\mathbb{R}^n$ by
     \begin{equation*}
         B_i:=\prod_{j=1}^{n}B_{i}^{(n)}=B_i^{(n)}\times\dots\times B_i^{(n)}.
     \end{equation*}
     Then $|B_i|=|B_i^{(n)}|^n$. It follows that $|B_1+B_2|=0$, $|B_1+B_3|=1$, $|B_2+B_3|=a$, and $|B_1+B_2+B_3|=b$. The rest of the proof goes the same as the case of dimension $n=1$ except we instead define $A_i:=\alpha_{13}^{\frac{1}{n}}B_i$. Then we see that 
     \begin{equation*}
         (0,0,0,0,0,\alpha_{13},\alpha_{23},\alpha_{123})\in\Lambda_n(3).
     \end{equation*}
     Finally, suppose there exists $\alpha>0$ for which $(0,0,0,0,0,\alpha,\alpha,\alpha)\in\Lambda_n(3)$. Then there exist compact sets $A_1,A_2,A_3$ in $\mathbb{R}^n$ such that $|A_j|=0$ for each $j\in[3]$, and $|A_1+A_3|=|A_2+A_3|=\alpha$ and $|A_1+A_2+A_3|=\alpha$. In particular, we have $|A_1+(A_2+A_3)|^{\frac{1}{n}}=|A_1|^{\frac{1}{n}}+|A_2+A_3|^{\frac{1}{n}}$. By the equality conditions for the Brunn-Minkowski-Lyusternik inequality, $\textup{card}(A_1)=1$. Then $|A_1+A_3|=|A_3|=0$, which is a contradiction. Therefore, $(0,0,0,0,0,\alpha,\alpha,\alpha)\notin\Lambda_n(3)$ for any $\alpha>0$. The fact that $\Lambda_{n}(3)$ is not closed now follows from observing that for each $k\geq2$ we have
     \begin{equation*}
         (0,0,0,0,0,1-\frac{1}{k},1-\frac{1}{k},1)\in\Lambda_{n}(3),
     \end{equation*}
     but the limit of this sequence is
     \begin{equation*}
         (0,0,0,0,0,1,1,1)\notin\Lambda_{n}(3).
     \end{equation*}
     In the case that $m>3$, just use the method described above to choose the sets $A_1,A_2$ and $A_3$, and set $A_j=\{0\}$ for each $j=4,\dots,m$.
\end{proof}

\newpage
\bibliography{bibliography}
\bibliographystyle{plain}

\end{document}